\documentclass[10pt]{article}
\usepackage{amssymb,amsmath,amsthm}
\usepackage{bm}
\usepackage{}
\usepackage{amsbsy}
\usepackage{psfrag}
\usepackage{graphics}
\usepackage{graphicx}
\usepackage{pgfplots}
\usepackage{relsize}
\usepackage{color}
\usepackage{xcolor}
\usepackage{mathtools}
\usepackage{pdfpages}
\usepackage{url}
 \usepackage{hyperref}
\usepackage{graphicx}
\usepackage{caption}
\usepackage{subcaption}
\usepackage{algpseudocode}
\usepackage{algorithm}
\usepackage{verbatim}
\usepackage{wrapfig}
\hypersetup{
	colorlinks=true,
	linkcolor=blue,
	filecolor=magenta,
	urlcolor=cyan,
}
\pgfplotsset{compat=1.18}

\theoremstyle{plain}
\newtheorem{theorem}{Theorem}[section]
\newtheorem{proposition}{Proposition}[section]
\newtheorem{definition}{Definition}[section]

\newtheorem{lemma}{Lemma}[section]

\newtheorem{example}[theorem]{Example}
\newcommand{\norm}[1]{\left\Vert #1 \right\Vert}
\newcommand{\brb}[1]{\Bigl( #1 \Bigr)}

\newtheorem{remark}[theorem]{Remark}
\date{}

\begin{document}
\allowdisplaybreaks[4]
%
%
%
 \numberwithin{equation}{section}
%

\author{Ramesh Chandra Sau \thanks{Department of Mathematics, Indian Institute of Technology, Bombay, India (\texttt{rcsau1994@gmail.com, rcsau@math.iitb.ac.in})} \and Luowei Yin \thanks{Department of Mathematics, The Chinese University of Hong Kong, Shatin, New Territories, Hong Kong, P.R. China (\texttt{lwyin@math.cuhk.edu.hk}).}}

\title{A Comparative Study of Neural Network Solvers for Second-order Boundary Value Problems}

\maketitle


\begin{abstract}
    Deep learning-based partial differential equation(PDE) solvers have received much attention in the past few years. Methods of this category can solve a wide range of PDEs with high accuracy, typically by transforming the problems into highly nonlinear optimization problems of neural network parameters. This work does a comparative study of several deep learning solvers proposed a few years back, including PINN, WAN, DRM, and VPINN. Numerical results are provided to make comparisons amongst them and address the importance of loss formulation and the optimization method. A rigorous error analysis for PINN is presented and a brief error estimate results for other methods have been discussed. Finally, we discuss the current limitations and bottlenecks of these methods. 
\end{abstract}

\section{Introduction}
Efficient numerical solvers for partial differential equations (PDEs) play an important role in the development of science and engineering. Classical solvers like finite element solver or finite difference solver often discretize the domain and/or the function space, after which the problem is assembled into matrix-vector systems and solved by well-established efficient matrix solvers. Recently, deep learning methods have received increasing attention in the PDE community. These new methods adopt neural networks as an approximator. In contrast with classical function approximators, the parametrization of a neural network is highly non-linear. With the merit of the powerful computation hardware, the neural networks can be efficiently evaluated or differentiated, hence becoming an alternative to classical approximations. Neural networks are born to be mesh-free, known to be universal approximations and have the potential to overcome the curse of dimensionality. The deep learning-based method for solving PDE was first introduced in the work by Raissi et.al. \cite{RaissiPerdolarisKarniadakis:2019} in 2019.  It is popularly known as physics-informed neural networks(PINN). This method is based on the residual minimization of the PDE and it only uses the strong form of the PDE. After that, there are many enrichments of PINN and weak formulation-based methods have been introduced such as, deep Ritz method(DRM)\cite{yu2018deep}, variational PINN(VPINN) \cite{berrone2022variational}, weak adversarial net(WAN) \cite{zang2020weak} operator-learning methods\cite{lu2019deeponet,li2020fourier}. 

The PINN takes advantage of the smoothness of neural networks. The loss function is constructed by the PDE residual augmented with penalty terms to impose the boundary and initial conditions. This formulation is intuitive and versatile for various PDEs. However, in practice, the PINN usually fails to accurately approximate the solution. A first insight is the difficulty brought by the penalty terms, which introduces a balance between the condition of problem and the violation of constraints. This is a common issue in penalty methods, and several variants of PINN are proposed to address it; Wang et al. \cite{wang2021understanding} analyzed the stiffness of the gradient flow and suggests that the penalty weight can be made adaptive by balancing the scale of gradient of each term in the loss function. In Lu et.al.\cite{lu2021physics}, they suggests to design the neural network architecture such that the boundary constraints are strictly satisfied. This variant is quite straightforward and effective, but constructing such neural networks can be difficult, especially when the geometry is complex. Another class of variants is motivated by the common practice of constrained optimization. In the work of Son et.al.\cite{son2023enhanced}, augmented Lagrangian approach is adopted. The loss function comprises both penalty and multiplier terms. In the classical analysis of constrained optimization, this formulation is known to reduce the ill-condition of penalty methods\cite{Bertsekas:1982}. In addition to the focus on penalty terms, there is another line of research that addresses the importance of sampling strategy\cite{gao2023failure,wu2023comprehensive,zeng2022adaptive}. These works place more emphasis on the high-dimensional application of PINNs. Indeed, to accurately evaluate the integrals in the loss function, quadratures become less effective in high dimensional setting, and the method of choice is usually Monte Carlo. The uniform distribution of the collocation points cannot efficiently capture the local structure of the solutions; hence, it is suggested to perform adaptive sampling according to some error estimators.

The loss function in PINN is constructed from strong formulation of PDEs. As a consequence, approximated solution via PINN can become inaccurate when solution possess less regularity. This issue can be improved by utilizing weak formulations. The DRM stands as a pioneering approach to leveraging deep learning for solving variational problems. Rooted in the classical Ritz method, which seeks to solve variational problems by minimizing a functional (known as the Ritz functional) over a predefined trial function space, the deep Ritz method takes a transformative leap by employing deep neural networks to represent the trial function. This departure from traditional methods allows for a more flexible and adaptive representation, as neural networks can capture intricate patterns and nonlinearities inherent in the problem domain. By optimizing the parameters of the neural network, the deep Ritz method effectively minimizes the Ritz functional, yielding solutions that exhibit enhanced accuracy and generalizability across a wide range of variational problems. The VPINN and WAN are also weak formulation-based methods to approximate a weak solution of PDE. 

The optimization problem for neural network parameters is the central task to find a good approximation of the solution, but it is highly non-convex and poorly understood. Despite this fact, there are several works that suggest some optimization algorithms that are empirically observed to be more suitable for PINN. First-order methods like SGD and ADAM \cite{kingma2014adam} are of choice in most neural network methods because the optimization of DNN is usually large in scale. Recently, the limited-memory BFGS (L-BFGS) algorithm has been found to obtain a more accurate solution than ADAM; see e.g. the numerical results in (hPINN OCP). A distinct property of the L-BFGS algorithm is its ability to handle large-scale problems with faster convergence than first-order methods. The performance of L-BFGS in PINN suggests that utilizing the curvature information of the loss function may help to improve the accuracy. Recently, natural gradient descent is found to be useful in the training of PINNs; see \cite{muller2023achieving,nurbekyan2023efficient}. Natural gradient descent is essentially a steepest gradient descent in function space. The descent direction in parameter space is chosen to best approximate the function gradient in the tangent space of model manifold. Under this viewpoint, this method is hence first-order. In the work of Nurbekyan et.al.\cite{nurbekyan2023efficient}, the authors derived the natural gradient in various function spaces and proposed a detailed algorithm to compute the parameter descent direction efficiently. The numerical results suggest that natural gradient descent is much more effective than steepest gradient descent in terms of both accuracy and computation time, as long as the neural network is relatively small and the inverse of Fisher information matrix can be quickly calculated or approximated. In the work of Muller et al.\cite{muller2023achieving}, the authors proposed optimizing the neural network in the Hessian manifold induced by the loss function, which leads to a Newton update in the function space. Their numerical results show a very high accuracy that breaks the accuracy saturation of PINN observed in many applications. A rigorous proof of convergence is lacking, but this work again indicates that utilizing the curvature information of the loss function is beneficial in the training of PINNs. A more comprehensive survey of recent optimization methods for deep learning can be found in \cite{abdulkadirov2023survey}. We should mention that although the understanding of neural network optimization is far from mature, an effective optimization algorithm is crucial to the success of PINN.

The advantage of PINN over classical methods is still limited to high-dimensional settings or to the cases where a mesh-free approach is strongly needed. In the low dimensional setting, solving the parameters in PINN takes significantly more time than classical methods. In \cite{grossmann2023can} the author compared PINN with FEM by numerical examples ranging from 1D to 3D. The results suggested that PINN is faster when extrapolation is needed, but its training always take more time than FEM. Solving the coefficients in FEM is a linear problem with many well-developed solvers, but such an efficient solver for neural networks is still missing.

In this article, we describe and compare four popular deep learning-based solvers: PINN, DRM, WAN, and VPINN. Numerically, we demonstrate how the PINN method differs from other weak formulation-based methods such as DRM, WAN, and VPINN. We show that while PINN fails to approximate non-smooth solutions, DRM, WAN, and VPINN manage to do so to some extent. WAN involves a maximization problem that lacks a unique solution, leading to the failure of the optimization algorithm. Conversely, VPINN is dimensionally restricted due to its dependence on a mesh. Additionally, DRM requires a symmetric bilinear form. Each of these methods has its own limitations, necessitating careful selection for optimal results. We provide a rigorous error analysis of the PINN method for a second-order elliptic PDE and briefly discuss the error analyses of DRM, WAN, and VPINN through several remarks in this article.

The rest of the paper is organized as follows. In Section \ref{sec:preli_dnns}, we describe some preliminaries on neural networks and PINN. In Section \ref{sec:dnn_solvers}, we describe some neural network-based solvers for partial differential equations. In Section \ref{sec:experiments}, we present the results of numerous numerical experiments conducted to evaluate the efficacy of the neural network-based solvers. Additionally, a comparative study of performance of solvers is meticulously conducted within this section, offering insights into their relative performance.  Section \ref{sec:error_analysis} is devoted to the error analysis of neural network-based solvers.

\section{Deep neural networks}\label{sec:preli_dnns}
Deep neural networks (DNNs) have been widely used due to their extraordinary expressivity for approximating functions, and are essential to the great success of deep learning. In this article, fully connected feedforward neural networks have been taken into consideration. Such neural networks are compositions of affine linear transformation and nonlinear activations. More specificly, fix $L\in\mathbb{N}$, and integers $\{n_\ell\}_{\ell=0}^L\subset \mathbb{N}$, with $n_0=d$ and $n_L=1$. A neural network $u_\theta:\mathbb{R}^d\to \mathbb{R}$, with $\theta\in\mathbb{R}^{N_\theta}$ being the DNN parameters, is defined recursively by
\begin{figure}
		\centering
		\includegraphics[width=15cm, height=8cm]{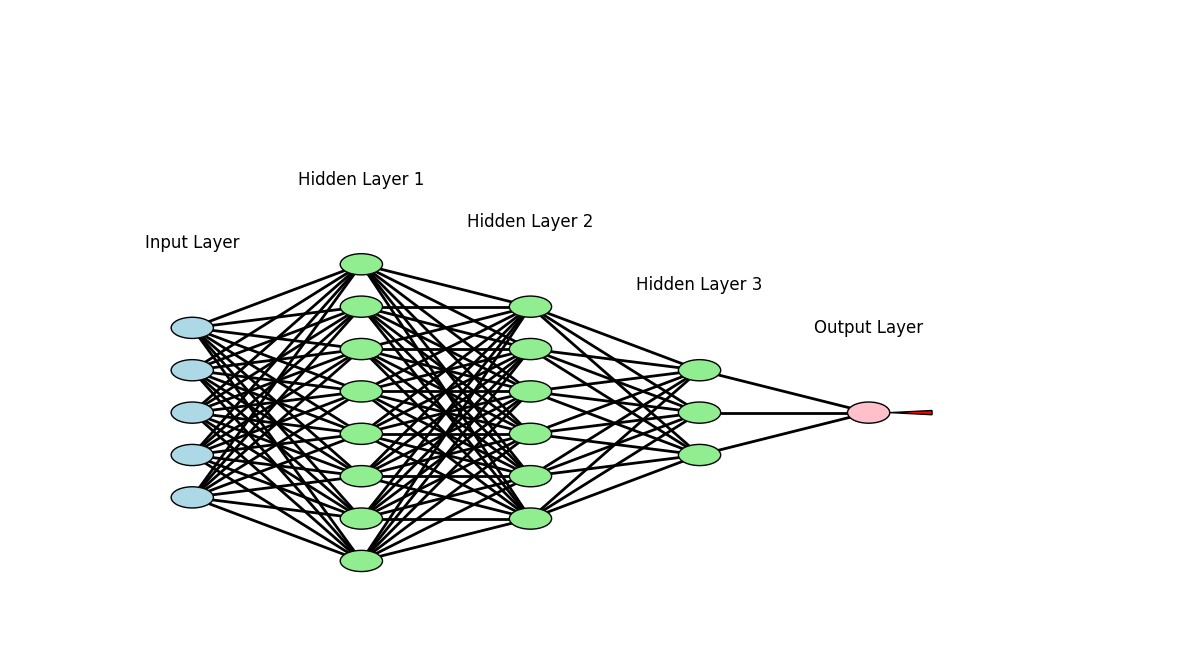}
  \caption{A four-layer neural network}
\end{figure}
\begin{align*}
  u^{(0)}(x) & = x,\\
  u^{(\ell)}(x) & =  \rho (W^{(\ell)}u^{(\ell-1)}+ b^{(\ell)}),\quad \ell = 1,\ldots, L-1,\\
  u_\theta:= u^{(L)}(x) &= W^{(L)} u^{(\ell-1)} + b^{(L)},
\end{align*}
where $W^{(\ell)}=[w_{ij}^{(\ell)}]\in \mathbb{R}^{n_\ell\times n_{\ell-1}}$, and $b^{(\ell)}=(b_i^{(\ell)})\in
\mathbb{R}^{n_\ell}$ are known as the weight matrix and bias vector at the $\ell$th layer, respectively. The mapping $\rho:\mathbb{R}\to\mathbb{R}$ is the nonlinear activation function.
$L$ is called the depth of the neural network, and $\omega :=\max(n_\ell, \ell=1,\ldots,L)$ is called
the width of the neural network. 
We denote by $\boldsymbol n_i$, $i=1,\ldots,L$, as the number of nonzero weights on
the first $i$ layers. Then $\boldsymbol{n}_L$ is the total number of nonzero weights. We use the notation
$\mathcal{N}_\rho(L,\boldsymbol{n}_L,R)$ to refer to the collection of DNN functions generated by $\rho$ of depth $L$, total number  $\boldsymbol n_L$ of nonzero weights and
each weight bounded by $R$. Throughout, we focus on two activation functions
$\rho$: hyperbolic tangent $\rho(t)=\frac{e^{t}-e^{-t}}{e^{t}+e^{-t}}$
and sigmoid $\rho(t)=\frac{1}{1+e^{-t}}$, which are popular in the literature on PINN.

\section{PDE solvers by deep neural networks}\label{sec:dnn_solvers}
Let $\Omega\subset\mathbb{R}^d$ be a smooth domain with boundary $\partial\Omega$ with $d\geq 1$. We consider the following general second-order elliptic partial differential equation
\begin{equation}\label{eqn:Poisson0}
    \left\{\begin{aligned}
         -\textit{div}(\mathcal{A}\nabla u) + \bm{\beta} \cdot \nabla u + cu &= f \quad \mbox{in }\Omega,\\
        u & = g\quad \mbox{on }\partial\Omega,
    \end{aligned}\right.
\end{equation}
where $\mathcal{A}= [a_{ij}]_{d\times d}$ with $a_{ij}\in C^1(\bar{\Omega}),$ and $\mathcal{A}$ satisfies the strict ellipticity condition, i.e., there exists a constant $\lambda>0$ such that $\xi^T\mathcal{A}\xi\geq \lambda |\xi|^2$ for all $\xi \in \mathbb{R}^d$, $\bm{\beta}\in [L^{\infty}(\Omega)]^d$, $c\in L^{\infty}(\Omega)$, and $f\in L^2(\Omega).$

\subsection{Physics informed neural networks(PINNs)}\label{ssec:pinn}
We describe physics informed neural networks (PINNs) of Rassi et al \cite{RaissiPerdolarisKarniadakis:2019}, for solving the elliptic boundary value problem \eqref{eqn:Poisson0}. PINN is based on the principle of PDE residual minimization. For problem \eqref{eqn:Poisson0}, the continuous loss $\mathcal{L}(u)$ is given by
\begin{equation}\label{cont:loss}
    \mathcal{L}(u) = \|-\textit{div}(\mathcal{A}\nabla u) + \bm{\beta} \cdot \nabla u + cu - f\|_{L^2(\Omega)}^2 + \alpha \|u-g\|_{L^2(\partial\Omega)}^2,
\end{equation}
where the penalty term $\|u-g\|_{L^2(\partial\Omega)}^2$ (with weight $\alpha>0$) is to approximately enforce the Dirichlet boundary condition. Let $U(\Omega),U(\partial\Omega)$ be the uniform distributions over sets $\Omega,\partial\Omega$, and $|\Omega|,|\partial\Omega|$ be their Lebesgue measures respectively, and $\mathbb{E}_{U(\Omega)},\mathbb{E}_{U(\partial\Omega)}$ be the corresponding expectations.
Then we can write the continuous loss \eqref{cont:loss} in terms of expectations: 
\begin{equation}\label{cont:loss1}
    \mathcal{L}(u) = |\Omega|\mathbb{E}_{X\sim U(\Omega)}[|(-\textit{div}(\mathcal{A}\nabla u) + \bm{\beta} \cdot \nabla u + cu - f)(X)|^2] +\alpha |\partial\Omega|\mathbb{E}_{Y\sim U(\partial\Omega)}[|(u-g)(Y)|^2].
\end{equation}
Let the sampling points $\{X_i\}_{i=1}^{n}$ and $\{Y_j\}_{j=1}^{m}$ be identically and independently distributed (i.i.d.), uniformly on the domain $\Omega$ and boundary $\partial\Omega$, respectively, i.e., $\mathbb{X}=\{X_i\}_{i=1}^{n}\sim U(\Omega)$ and $\mathbb{Y}=\{Y_j\}_{j=1}^{m}\sim U(\partial\Omega)$. Now we replace $u$ by $u_\theta$ in \eqref{cont:loss1} and apply the Monte Carlo method with the above interior and boundary sample points to obtain the following empirical loss $\widehat{\mathcal{L}}(u)$:
\begin{equation}\label{loss:pinn}
    \widehat{\mathcal{L}}(u) = \frac{|\Omega|}{n}\sum_{i=1}^{n}|(-\textit{div}(\mathcal{A}\nabla u) + \bm{\beta} \cdot \nabla u + cu - f)(X_i)|^2+ \frac{\alpha |\partial\Omega|}{m}\sum_{j=1}^{m}|(u-g)(Y_j)|^2 .
\end{equation}

Define PINN formulation of the problem \eqref{eqn:Poisson0} as follows: find $u_{\theta^*}\in \mathcal{U}:=\mathcal{N}_\rho(L,\boldsymbol{n}_L,R)$ such that
\begin{align}\label{PINN:formulation}
 u_{\theta^*}\in \underset{u_{\theta} \in \mathcal{U}}{\text{argmin}} \quad \widehat{\mathcal{L}}(u_{\theta}).
\end{align}

\begin{algorithm}
\caption{PINN}\label{alg:pinn}
\begin{algorithmic}[hbt]
\State Input: $n$,$m$, $\alpha$
\State Initialize: $\theta^0$ 
\State Set $k=0$
\State Sample collocation points $\{X_i\}_{i=1}^n$ in the interior($\Omega$), and $\{X_i\}_{i=1}^n$ on the boundary($\partial\Omega$).
\While{$k\leq N$} 
    \State $\theta^{k+1}\gets \theta^k - \tau_\theta \nabla_{\theta}\hat{\mathcal{L}}(\theta^k),$ compute $\nabla_\theta\hat{\mathcal{L}}(\theta^k)$ from \eqref{loss:pinn}
\EndWhile
\State Output: PINN solution $u_{\theta^{N+1}}.$
\end{algorithmic}
\end{algorithm}

\begin{remark}
 Note that the resulting optimization problem \eqref{PINN:formulation} has a solution due to the box constraint on the DNN parameters $\theta$, i.e., $|\theta|_{\ell^\infty}\leq R$ for suitable $R$, which induces a compact set in $\mathbb{R}^{\mathbf{n}_L}$. Meanwhile the empirical loss $\widehat{\mathcal{L}}(u_{\theta})$ is continuous in $\theta$, when the activation function $\rho$ is smooth. In practice, the minimizer $u_{\theta^*}$ is commonly achieved by off-shelf optimizers, e.g., limited memory BFGS \cite{ByrdLu:1995} and Adam \cite{KingmaBa:2015}, which are implemented in many public software frameworks, e.g., PyTorch or TensorFlow.  To evaluate the loss $\widehat{\mathcal{L}}(u_{\theta})$, one requires computing the derivative of the network outputs $u_{\theta}$ with respect to the input $x$ (i.e., spatial derivatives), and to apply a first-order optimizer, one requires computing the derivative of the loss $\widehat{\mathcal{L}}(u_{\theta})$ with respect to the DNN parameter $\theta$. Both can be realized using automatic differentiation efficiently \cite{Baydin:2018}, e.g., \texttt{torch.autograd} in PyTorch. Given a (local) minimizer $\theta^*$, the corresponding DNN approximations of the state $u\in\mathcal{U}$ is given by $u_{\theta^*}$. The theoretical analysis of the PINNs has been carried out in \cite{JiaoLai:2022cicp,ShinDarbon:2020,MishraMolinaro:2023,LuChenLu:2021}.   
\end{remark}

\subsection{Deep Ritz method(DRM)}\label{subsec:drm}
The deep Ritz method is a deep learning-based numerical algorithm for solving variational problems was proposed by Weinan et al \cite{EYu:2018}. It is inspired by the classical Ritz method, which is the solution of a variational problem by minimizing a functional(Ritz functional) over a trial function space. The deep Ritz method uses a deep neural network to represent the trial function, and optimizes the network parameters to minimize the Ritz functional. For this method, we slightly modify the equation \eqref{eqn:Poisson0} with $\bm{\beta}=\mathbf{0}$ to have a symmetric bilinear form. Also, we assume that the reaction coefficient $c(x)\geq 0$ and $g\in H^{1/2}(\partial \Omega)$ for the existence of the weak solution. Thus the continuous loss of the deep Ritz method for the model problem \eqref{eqn:Poisson0}, is as follows:
\begin{align}\label{DRM_cLoss}
     \mathcal{L}(u) =  \int_{\Omega} (\nabla u^{T} \mathcal{A} \nabla u +  c u^2 - f u) + \alpha \int_{\partial \Omega} |u-g|^2.
\end{align}
We discretise \eqref{DRM_cLoss}, to obtain an empirical loss $\hat{\mathcal{L}}$ using the Monte Carlo method. The empirical loss for DRM is given by
\begin{align}\label{drm_emploss}
    \hat{\mathcal{L}}(u) = \frac{|\Omega|}{n} \sum_{i=1}^n \brb{(\nabla u^{T} \mathcal{A} \nabla u + c u^2-fu)(X_i)} +\alpha \frac{|\partial\Omega|}{m} \sum_{i=1}^m |(u-g)(Y_i)|^2.
\end{align}
Hence the DRM formulation of the problem \eqref{eqn:Poisson0} is as follows:
\begin{align}\label{DRM:formulation}
 u_{\theta^*}\in \underset{u_{\theta} \in \mathcal{U}}{\text{argmin}} \quad \widehat{\mathcal{L}}(u_{\theta}).
\end{align}

 The algorithm for DRM is similar to Algorithm \ref{alg:pinn}, one needs to use the DRM loss function \eqref{drm_emploss}.

\subsection{Variational PINN (VPINN)}\label{subsec:vpinn}
The Variational PINN relies on the weak formulation of the PDE \eqref{eqn:Poisson0}, and it was first introduced by Berron et al \cite{berrone2022variational}. Let $\Omega$ be a domain in $\mathbb{R}^d$, $g\in H^{1/2}(\partial\Omega),$ $\bm{\beta}\in [W^{1,\infty}(\Omega)]^n$, and $c-\frac{1}{2} \nabla\cdot\bm{\beta}\geq 0.$ 
Let $u_g\in H^1(\Omega)$ be an extension of $g$ in $\Omega$. 
Then the weak formulation of \eqref{eqn:Poisson0} reads as: find $w\in H^1_0(\Omega)$ such that
\begin{align}\label{cont_weakform}
    a(w,v) = (f,v) - a(u_g,v)\quad \forall\; v\in H^1_0(\Omega),
\end{align}
where $a(p,q)=\int_{\Omega} \nabla p^{T}\mathcal{A}\nabla q + (\bm{\beta}\cdot \nabla p) q+ c pq$. The weak solution of \eqref{eqn:Poisson0} is defined as $u = w+u_g$. In VPINN, we use neural networks to approximate the trial function and finite element functions in test functions. Let $\mathcal{T}_h$ be a mesh for the domain $\Omega$ and $V_h :=\{v_h\in L^2(\Omega)| v_h|_T\in \mathbb{P}_k(T),\; \forall T \in \mathcal{T}_h\} $ be the approximate test function space. Our aim is to find a neural network which approximately solves: find $w_\theta\in H_0^1(\Omega)$ such that
\begin{align}\label{disc_weakform}
    a_h(w_\theta,v_h) = (f,v_h) - a_h(u_g,v_h)\quad \forall\; v_h\in V_h,
\end{align}
where $a_h(p,q) = \sum_{T\in \mathcal{T}_h}\int_{T}\nabla p^{T}\mathcal{A}\nabla q + (\bm{\beta}\cdot \nabla p) q+ c pq \; dx$.

In order to handle the problem \eqref{disc_weakform}, we introduce a basis $\{\phi_i\}_{i=1}^{n_b}$ in $V_h$, and let us define the residuals
\begin{align}\label{residuals}
    r_{h,i}(w_\theta) := a_h(w_\theta, \phi_i) - (f, \phi_i) + a_h(u_g,\phi_i), \quad \text{for}\; i = 1,2,3,\cdots,  n_b, 
\end{align}
as well as the loss function
\begin{align}
    \mathcal{L}_h(w_\theta) = \sum_{i=1}^{n_b} \alpha_i r_{h,i}^2(w_\theta),
\end{align}
where $\alpha_i>0$ are suitable weights. To impose the zero boundary of $w_\theta$, we set $w_\theta = \tilde{w}_\theta \phi_\eta$, where $\phi_eta$ is either a analytical function or pre-trained neural network that approximately vanished on boundary. Finally, we use Monte Carlo method to compute the integrals in $r_{h,i}(w_\theta)$ and obtain $\hat{r}_{h,i}(w_\theta)$. Then we search for a global minimum of the empirical loss function: find $w_{\theta*}\in \mathcal{U}$ such that 
\begin{align}\label{emp_loss_vpinn}
    w_{\theta*}\in \underset{w_{\theta} \in \mathcal{U}}{\text{argmin}} \quad\hat{\mathcal{L}}_h(w_\theta) := \sum_{i=1}^{n_b} \alpha_i \hat{r}_{h,i}^2(w_\theta),
\end{align}
Hence, we define $u=w_{\theta*}+u_g$ as an approximation of the weak solution of \eqref{eqn:Poisson0}. As one can see that the VPINN method needs basis functions of the finite element space $V_h.$ Thus computing basis functions on a certain mesh in high dimension ($\mathbb{R}^d, \; d\geq 4$) is infeasible. Therefore for VPINN we restrict ourselves to two or three-dimensional domains.

\begin{algorithm}
\caption{Variational PINN}\label{alg:vpinn}
\begin{algorithmic}[hbt]
\State Input: Import mesh $\mathcal{T}_h$, construct $\mathbb{P}_k$ polynomial basis for the test space $V_h$, $\alpha_i$\; $i = 1,2,3, \cdots ,n_b$
\State Initialize: Trial neural network parameter $\theta^0$ 
\State Set $k=0$
\While{$k\leq N$} 
    \State $\theta^{k+1}\gets \theta^k - \tau_\theta \nabla_{\theta}\hat{\mathcal{L}}(\theta^k),$ compute $\nabla_\theta\hat{\mathcal{L}}(\theta^k)$ from \eqref{emp_loss_vpinn}
\EndWhile
\State $u^{g}_{\phi*}:= \text{argmin}_{\phi} \frac{1}{M} \sum_{i=1}^N|u_\phi(x_i) - u_g(x_i)|^2,$ be a neural network approximates $u_g$
\State Output: PINN solution $u_{\theta^{N+1}}+u^{g}_{\phi*}.$
\end{algorithmic}
\end{algorithm}

\subsection{Weak adversarial network(WAN)}\label{ssec:wan} Weak adversarial network (WAN) was introduced by Yaohua Zang et al \cite{zang2020weak} to approximate the weak solution of a PDE. We take the same assumption on data as in subsection \ref{subsec:vpinn}. WAN works by transforming the problem of finding the weak solution of PDEs into a minimization problem of an operator norm, and using two neural networks to represent the weak solution and the test function. The two networks are trained in an adversarial way, meaning that they try to minimize and maximize the loss function respectively. By doing so, the WAN method can find an approximate weak solution of PDEs that satisfies the boundary conditions and the weak formulation. Multiplying the \eqref{eqn:Poisson0} by a test function $\phi \in H^1_0(\Omega)$ and integration by parts:
\begin{align}\label{weak_form}
    \langle A[u],\phi \rangle &:= \int_{\Omega} \nabla u^{T} \mathcal{A} \nabla \phi + \int_{\Omega} (\bm{\beta} \cdot \nabla u) \phi + \int_{\Omega} c u \phi - \int_{\Omega} f \phi\\
u &= g \quad \text{on}\; \partial\Omega.
\end{align}
We can consider $A[u]: H^1_0(\Omega)\longrightarrow \mathbb{R}$ as a linear operator such that $A[u](\phi):=\langle A[u],\phi \rangle$ as defined in \eqref{weak_form}. Then the operator norm of $A[u]$ is defined as 
\begin{align}\label{operator_norm}
    \norm{A[u]}_{op} := \sup_{\phi\in H^1_0(\Omega)\setminus \{0\}} \frac{|\langle A[u],\phi \rangle|}{\norm{\phi}_{H^1(\Omega)}}.
\end{align}
Therefore, $u$ is a weak solution of \eqref{eqn:Poisson0} with the boundary condition $u=g$ is satisfied on $\partial \Omega$ if and only if $u$ solves the following minmax problem:
\begin{align}\label{cont_wan}
 \min _{u \in H^{1}}\|A[u]\|_{o p}^{2} = \min _{u \in H^{1}} \sup _{\phi \in H_{0}^{1}}|\langle A[u], \phi\rangle|^{2} /\|\phi\|_{H^1(\Omega)}^{2}   
\end{align}

The formulation \eqref{cont_wan} inspires an adversarial approach to find the weak solution of \eqref{eqn:Poisson0}. More specifically, we seek for the function $u_{\theta}$, realized as a deep neural network with parameter $\theta$ to be learned, such that $A\left[u_{\theta}\right]$ minimizes the operator norm \eqref{cont_wan}. On the other hand, the test function $\phi$, is a deep adversarial network with parameter $\eta$, also to be learned. We define the interior part of the loss function by
\begin{align}\label{loss_wan_int}
   \mathcal{L}_{\mathrm{int}}(u_{\theta}, \phi_\eta) := |\langle A[u_{\theta}], \phi_\eta \rangle|^2 /\|\phi_\eta\|_{H^1(\Omega)}^2. 
\end{align}

In addition, the weak solution $u_{\theta}$ also needs to satisfy the boundary condition as in \eqref{eqn:Poisson0}. Thus we need to add the following boundary loss:
\begin{align}\label{loss_wan_bdry}
   \mathcal{L}_{\text {bdry}}(u_{\theta}) := \norm{u_{\theta}-g}^2_{L^2(\Omega)} 
\end{align}
Thus we can compute loss functions in \eqref{loss_wan_int} and \eqref{loss_wan_bdry} by Monte Carlo method and produce empirical loss components $\hat{\mathcal{L}}_{\mathrm{int}}(u_{\theta}, \phi_\eta)$ and $\hat{\mathcal{L}}_{\text {bdry}}(u_{\theta})$. Hence the total empirical loss function is the weighted sum of the two empirical loss components $\hat{\mathcal{L}}_{\mathrm{int}}(u_{\theta}, \phi_\eta)$ and $\hat{\mathcal{L}}_{\text {bdry}}(u_{\theta})$. Thus the WAN method is to solve the following minmax problem:
\begin{align}\label{wan_minimax}
   \min _{u_{\theta}} \max _{\phi_\eta} \hat{\mathcal{L}}(u_{\theta}, \phi_\eta), \quad \text { where } \hat{\mathcal{L}}(u_{\theta}, \phi_\eta) := \hat{\mathcal{L}}_{\mathrm{int}}(u_{\theta}, \phi_\eta)+\alpha \hat{\mathcal{L}}_{\mathrm{bdry}}(u_{\theta}), 
\end{align}
where $\alpha>0$ is user-chosen balancing parameter. 

The test function in WAN belongs to $H_0^1(\Omega)$. To impose the zero boundary, we follow the setting in the original work\cite{zang2020weak} by setting
 \[
\phi_\eta = w\cdot \hat{\phi_\eta}.
 \]
 where $\hat{\phi_\eta}$ is the neural network, and the function $w$ is either a prescribed and fixed function that strictly satisfy the zero boundary condition, or another neural network pretrained to approximate the boundary data and is kept fixed during training $\hat{\phi_\eta}$. When finding an analytic function described in former condition becomes hard, the latter can be used. The minimax problem \eqref{wan_minimax} is solved in alternating manner: $N_\eta$ steps of optimization algorithm is applied to inner maximization, and then $N_\theta$ steps is applied to outer minimization. 
\begin{algorithm}
\caption{Weak Adversarial Network}\label{alg:wan}
\begin{algorithmic}[hbt]
\State Initialize: $\eta^0$, $\theta^0$. 
\While{$i\leq N_{\text{opt}}$} 
    \State $\eta^i_0\gets \eta^i,\theta^i_0\gets\theta^i$
    \While{$j\leq N_\eta$}
    \State $\eta^i_{j+1} \gets \eta^i_{j} + \tau_\eta \nabla_{\eta}\hat{\mathcal{L}}(\eta^i_j,\theta^i)$
    \EndWhile 
    \State $\eta^{i+1}\gets \eta^i_{N_\eta}$
    \While{$j\leq N_\theta$}
    \State $\theta^i_{j+1}\gets \theta^i_{j} - \tau_\theta \nabla_{\theta}\hat{\mathcal{L}}(\eta^{i+1},\theta^i_j)$
    \EndWhile 
    \State $\theta^{i+1}\gets \theta^{i}_{N_\theta}$
\EndWhile
\end{algorithmic}
\end{algorithm}

The loss function \eqref{wan_minimax} indicates that the maximization problem in the inner loop does not possess a unique solution. This ill-posedness will lead to the failure of the optimization algorithm. Empirically, this method works with fine-tuned hyperparameters: boundary weight should be in the correct scale, and the minimizing and maximizing steps should be balanced. First-order methods, for instance, SGD, AdaGrad, or ADAM, are more suitable than second-order methods, for instance Newton method and BFGS, because using curvature in this situation can quickly drive the test function to become singular. This ill-posedness also motivates some regularization techniques, one of which is a stabilized version of WAN\cite{bertoluzza2023best}.

In the following section, we perform some numerical experiments to illustrate all the DNN-based methods from \eqref{ssec:pinn} to \eqref{ssec:wan}, and also do some comparisons among them.

\section{Numerical Experiments}\label{sec:experiments}
We split our numerical experiments into two sections, one is low dimension and another is high dimension. It is worth noting that, due to a lack of effective training strategy, the presented deep learning-based methods are still not competitive with classical methods. A detailed comparison between PINN and FEM can be found in \cite{grossmann2023can}. The results therein indicate that the strength of PINN lies in problems with 3 dimensions or higher. In the following experiments, our comparison will be limited to deep learning-based methods.
\subsection{Low dimension examples}
\begin{example}[\bf PINN in 2D]\label{example:pinn2d}
We consider the general elliptic equation \eqref{eqn:Poisson0} in the two dimensional domain $\Omega=(0,1)^2$, with the coefficient matrix $\mathcal{A} =
\begin{bmatrix}
  4 & -1 \\
  -1 & 3
\end{bmatrix}$, $\bm{\beta} = [1 \; 2]^T,$ $c=2$, and the known exact solution $u = \sin(\pi x_1) \sin(\pi x_2).$ The source $f$ has been computed as follows $f = -\textit{div}(\mathcal{A}\nabla u) + \bm{\beta} \cdot \nabla u + cu,$ and the boundary data $g = u \; \text{on}\; \partial \Omega$.
\end{example}

 We use a DNN architecture with 4 hidden layers and 30 neurons in each layer. When formulating the empirical loss $\widehat{\mathcal{L}}(u_{\theta})$, we take $n=4000$ points i.i.d. from $U(\Omega)$ for the PDE residual, and $m=1000$ points i.i.d. from $U(\partial \Omega)$, for boundary residual and boundary weight $\alpha = 1000$. These settings are chosen by trial and error.

 We minimize the loss $\widehat{\mathcal{L}}(u_{\theta})$ using ADAM provided by the PyTorch library \texttt{torch.optim} (version 2.0.0). We use 5000 ADAM iterations to minimize the loss function, the learning rate = $1e-4$. The exact solution and computed neural network solution and pointwise errors are shown in Figure \ref{fig:2D_Poisson} and loss history has been shown in Figure \ref{fig:2D_loss_hist}. We compute the $L^2-$ error, which is $0.025051$, and the $L^2-$ relative error is $5\%$, the $L^2-$ relative error is defined as follows: 
 \begin{align*}
     \norm{u-u_\theta}_{*} := \frac{\norm{u-u_\theta}_{L^2(\Omega)}}{\norm{u}_{L^2(\Omega)}}.
 \end{align*}
 \begin{figure}[h]
    \centering
    \setlength{\tabcolsep}{0pt}
    \begin{tabular}{ccc}
            \includegraphics[width=0.24\textwidth]{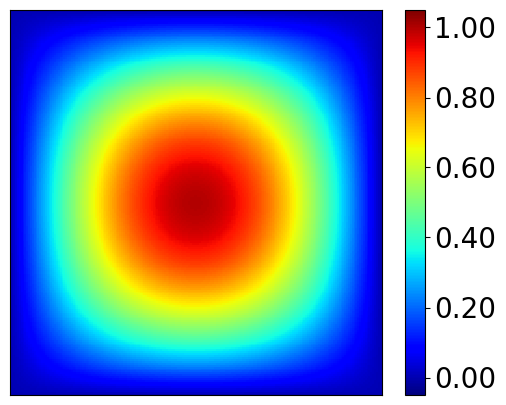}
            &\includegraphics[width=0.24\textwidth]{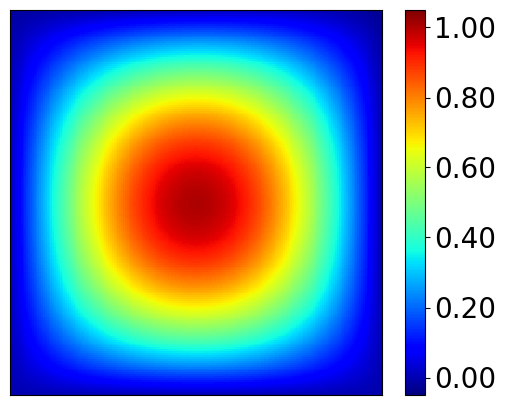}
            &\includegraphics[width=0.24\textwidth]{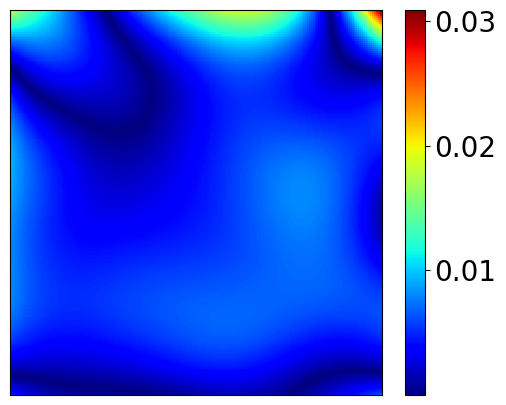}\\
            (a)  & (b) & (c) 
    \end{tabular}
    \caption{(a) Exact solution($u$) (b) Neural network approximate solution($u_{\theta}$) (c) pointwise error $|u_{\theta}(x)- u(x)|$} for Example \ref{example:pinn2d}
    \label{fig:2D_Poisson}
\end{figure}

\begin{figure}
    \centering
    \includegraphics[width=0.50\textwidth]{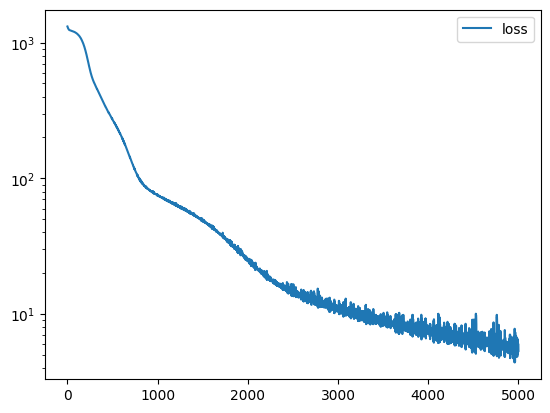}
    \caption{Loss history($\hat{\mathcal{L}}(u_{\theta})$), for Example \ref{example:pinn2d}}
    \label{fig:2D_loss_hist}
\end{figure}

Thus we see from Example \ref{example:pinn2d} that PINN works well for strong solutions, i.e., it approximates strong solutions very well. In the next example, We show that PINN is ineffective at approximating weak solutions, while DRM and VPINN are quite good at doing so.
\begin{example}[\bf DRM and VPINN in 2D]\label{example:Ldomain}
We consider the elliptic equation \eqref{eqn:Poisson0} in the L-shaped domain $\Omega=(0,1)^2\backslash [0,1)\times (-1,0]$ with the coefficient matrix $\mathcal{A} =
\begin{bmatrix}
  1 & 0 \\
  0 & 1
\end{bmatrix}$, $\bm{\beta} = [0 \; 0]^T,$ $c=0$. In radial coordinates, the source is given by $f(r,\theta)=\frac{2}{27}r^{\frac{-5}{3}}\sin(\frac{2}{3}\theta)-\frac{4}{9}r^{-1}\sin(\frac{2}{3}\theta)$. This PDE has a ground truth solution $y^*(r,\theta)=r^{\frac{2}{3}}\sin(\frac{2}{3}\theta)$.
\end{example}
For DRM, we use a DNN architecture with 4 hidden layers and 30 neurons in each layer. When formulating the empirical loss $\widehat{\mathcal{L}}(u_{\theta})$, we take $n=5000$ points i.i.d. from $U(\Omega)$ for the PDE residual, and $m=1000$ points i.i.d. from $U(\partial \Omega)$, for boundary residual and boundary weight $\alpha = 100$. These settings are chosen by trial and error. We minimize the loss $\widehat{\mathcal{L}}(u_{\theta})$ using ADAM provided by the PyTorch library. We use 5000 ADAM iterations to minimize the loss function, the learning rate = $1e-3$, and are multiplied by 0.1 after 3000 and 4000 iterations.  Moreover, we use the same architecture for PINN.

For VPINN, we take Delaunay triangulation with 2113 nodes and 4096 elements 
in total. The test functions are piecewise linear functions. We use another two pre-trained neural networks to enforce the boundary condition, each having 2 hidden layers with width 30. They are trained on the known boundary data by supervised learning until the error reaches 0.1\%.
For all the three methods, we employ ADAM to optimize the bias and weights in neural networks. The learning rate setting is the same as DRM. 

The results are displayed in Table \ref{fig:2dLshape}.
The PINN solution has relatively large errors spreading over the whole domain, 
while the errors of DRM and VPINN mostly concentrate on corner singularity. The $L^2$ relative errors for PINN, DRM and VPINN are 16.98\%, 4.73\% and 5.22\%, respectively.
Thus the $L^2$ errors of DRM and VPINN are both lower than PINN, indicating that a weak formulation can indeed fascillate neural networks to solve weaker solutions.

\begin{figure}[hbt!]
    \centering
    \begin{tabular}{cccc}
        \includegraphics[width=0.19\textwidth]{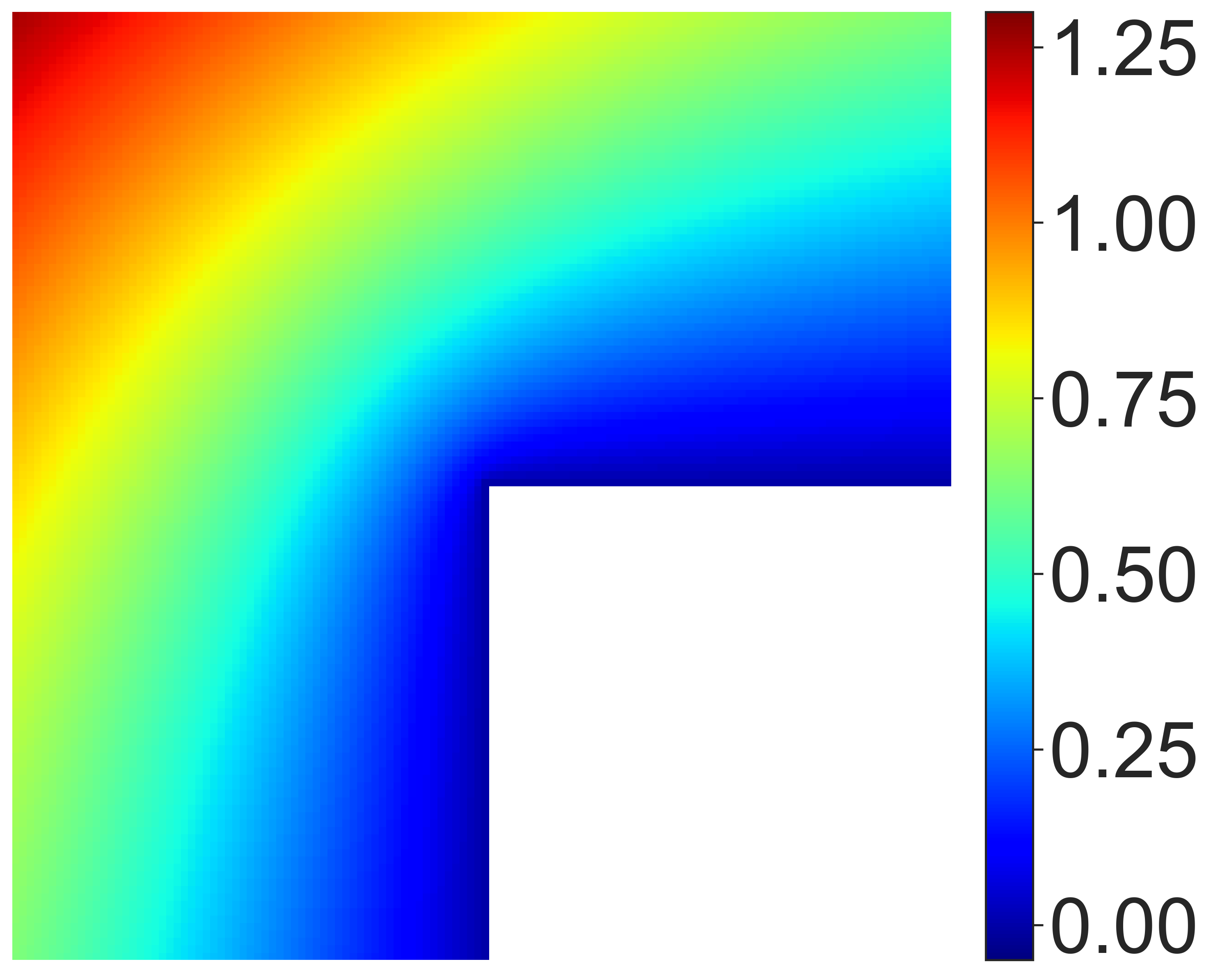} &
        \includegraphics[width=0.19\textwidth]{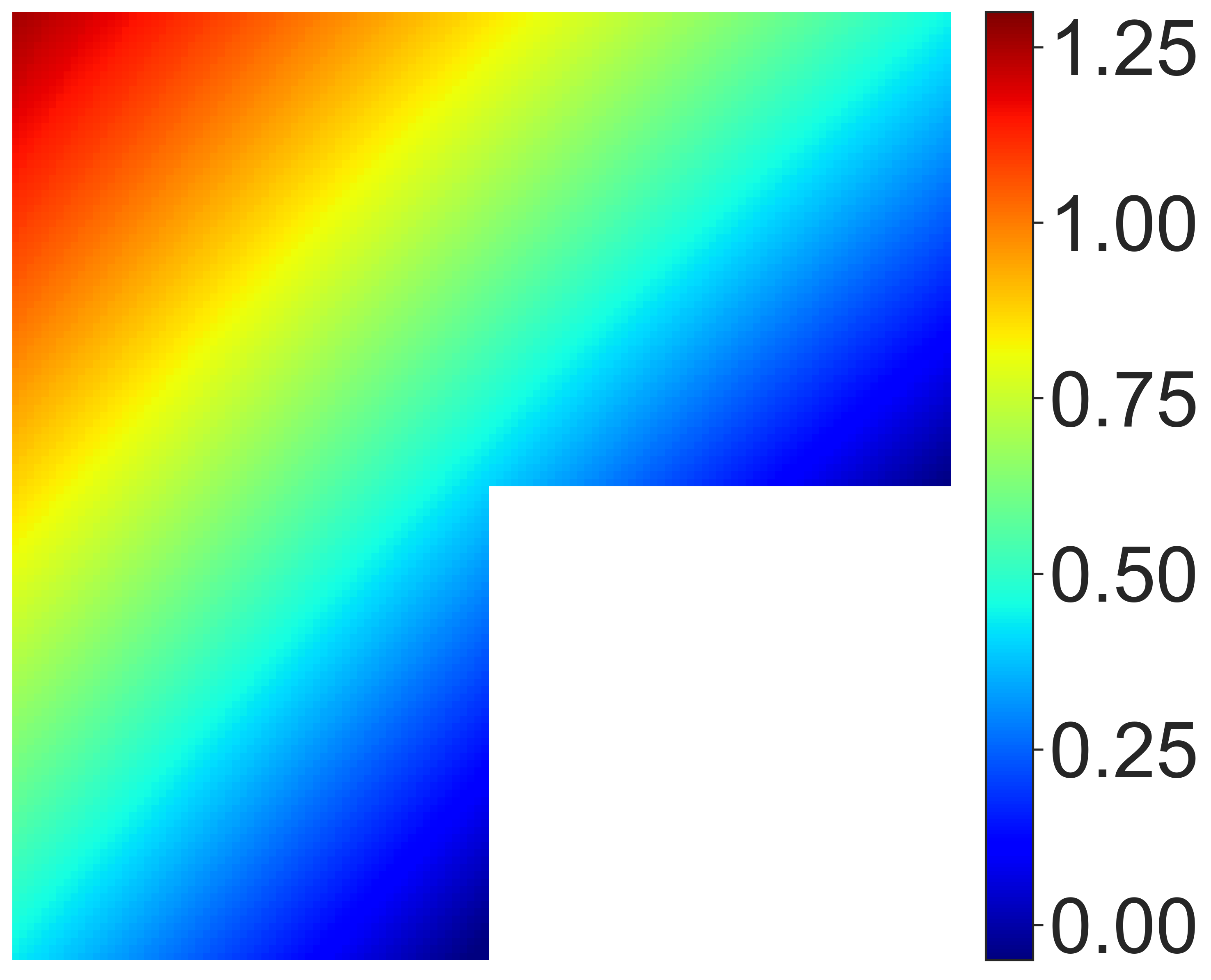} &
        \includegraphics[width=0.19\textwidth]{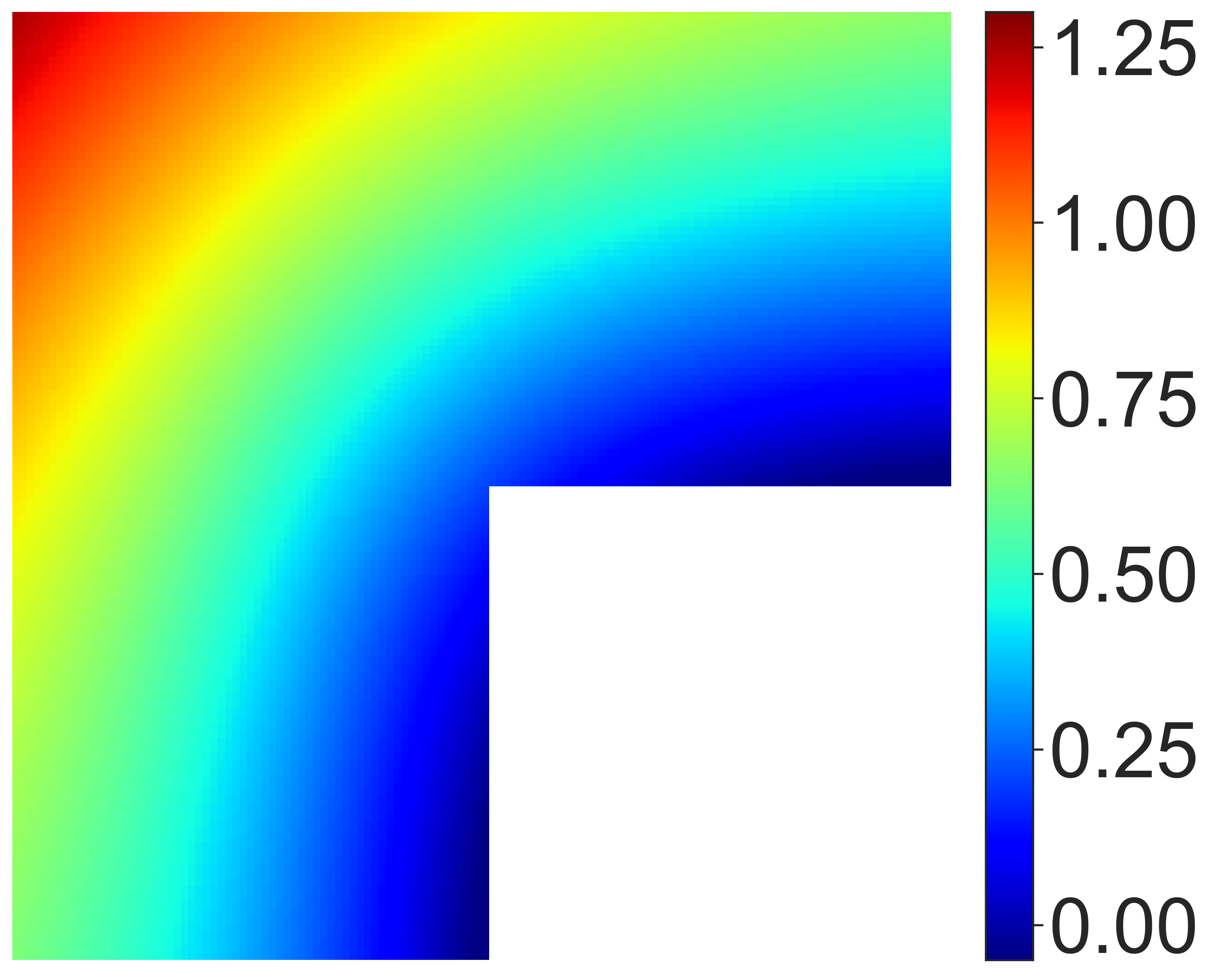} &
        \includegraphics[width=0.19\textwidth]{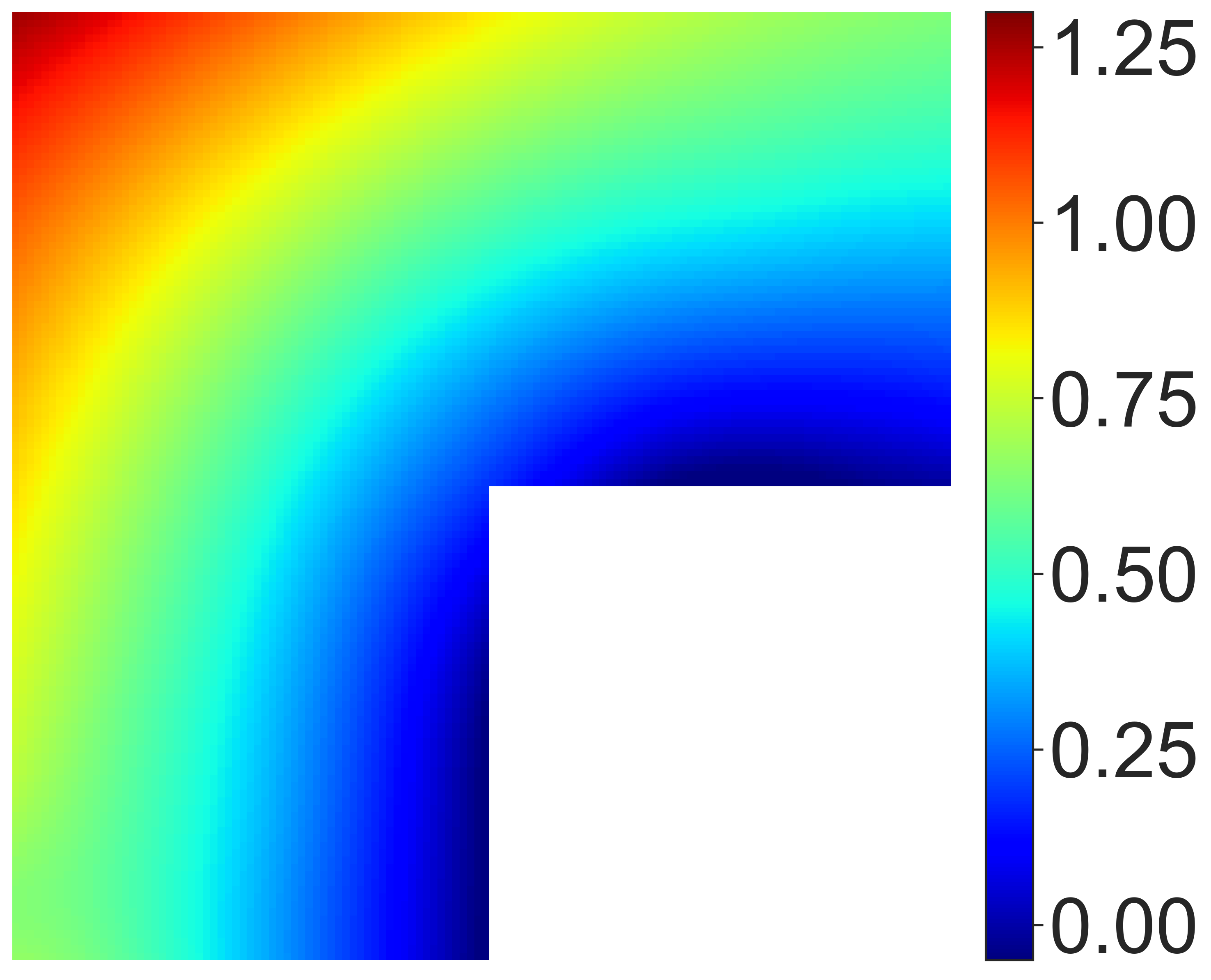} \\
         &
         \includegraphics[width=0.19\textwidth]{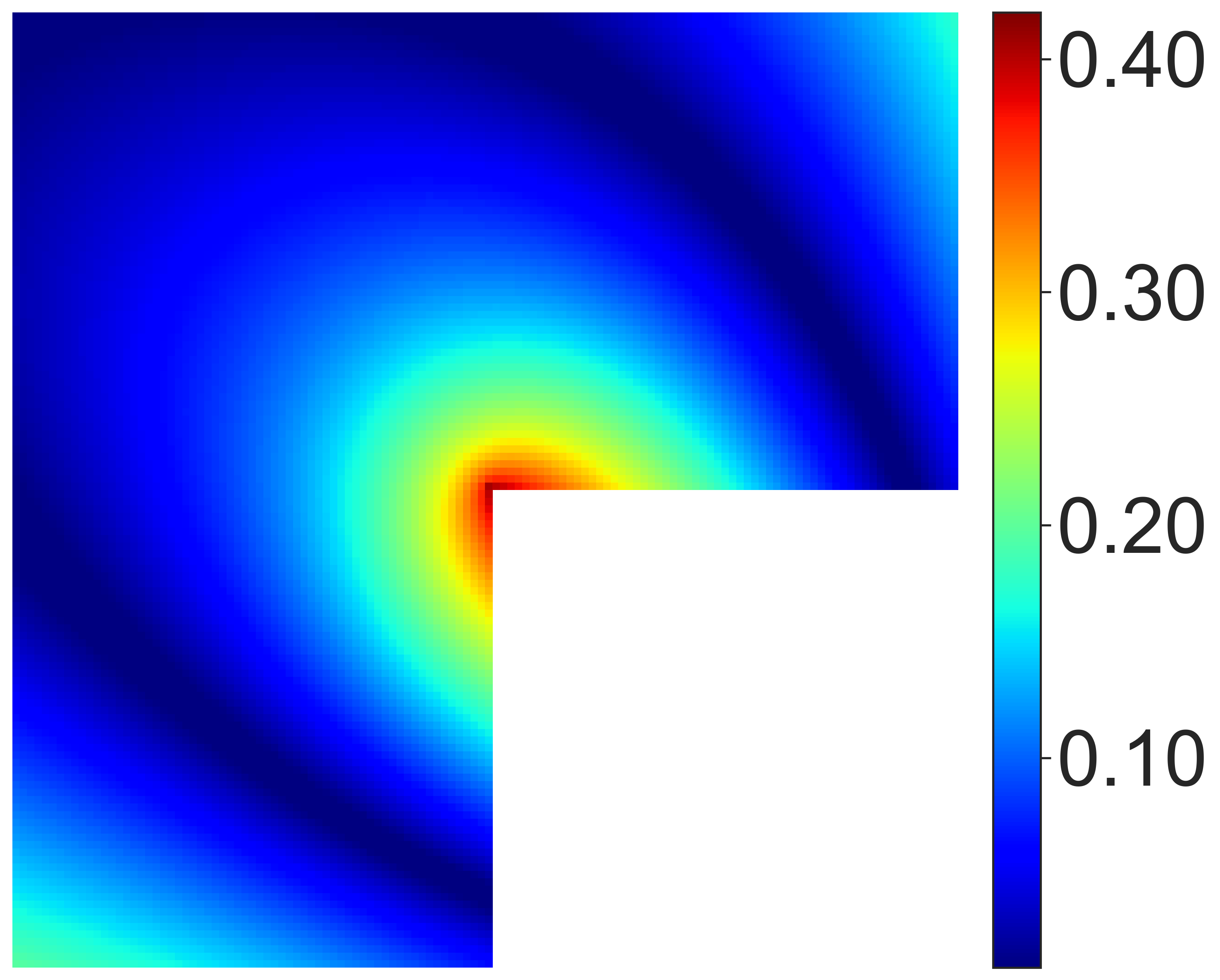} &
         \includegraphics[width=0.19\textwidth]{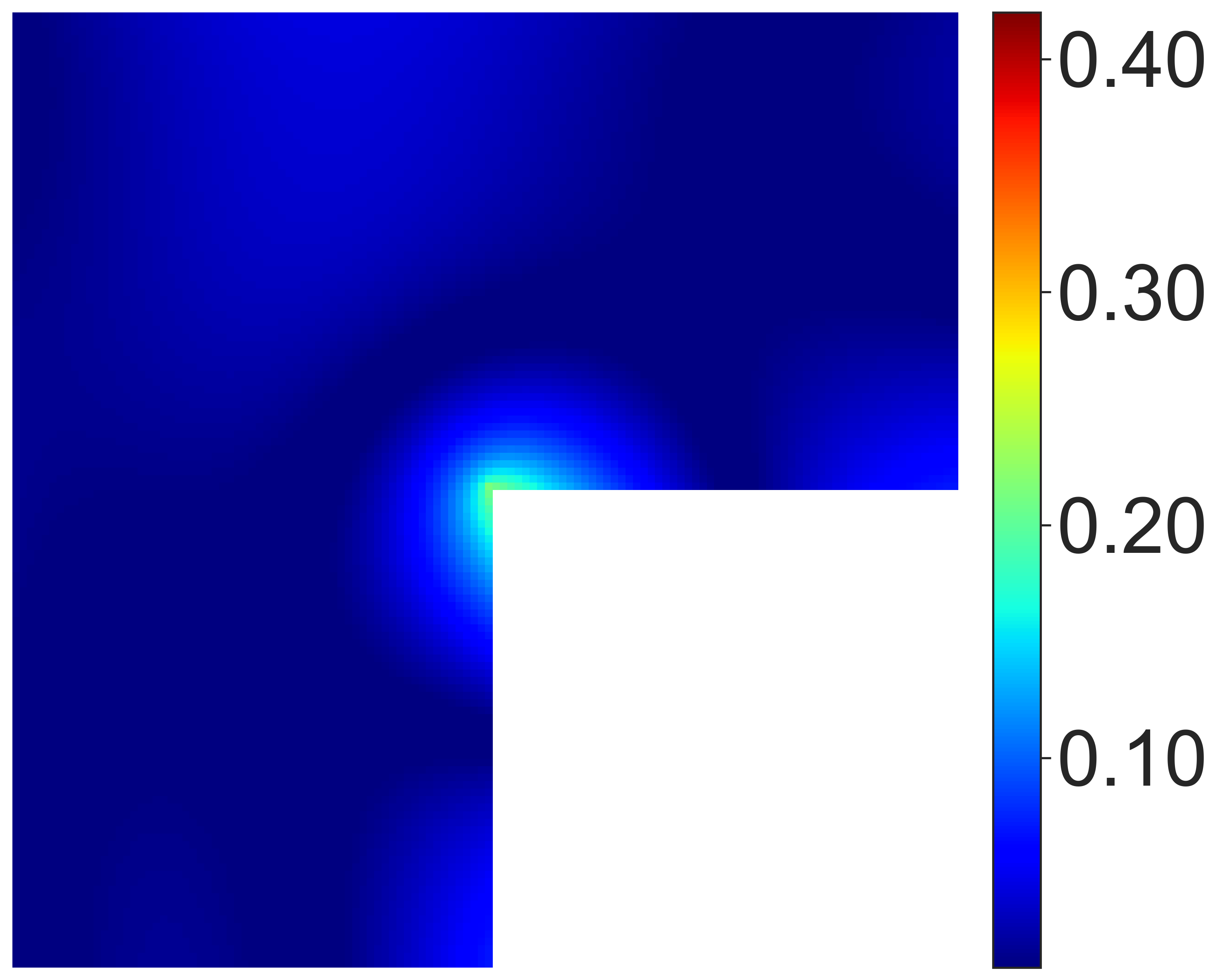} &
         \includegraphics[width=0.19\textwidth]{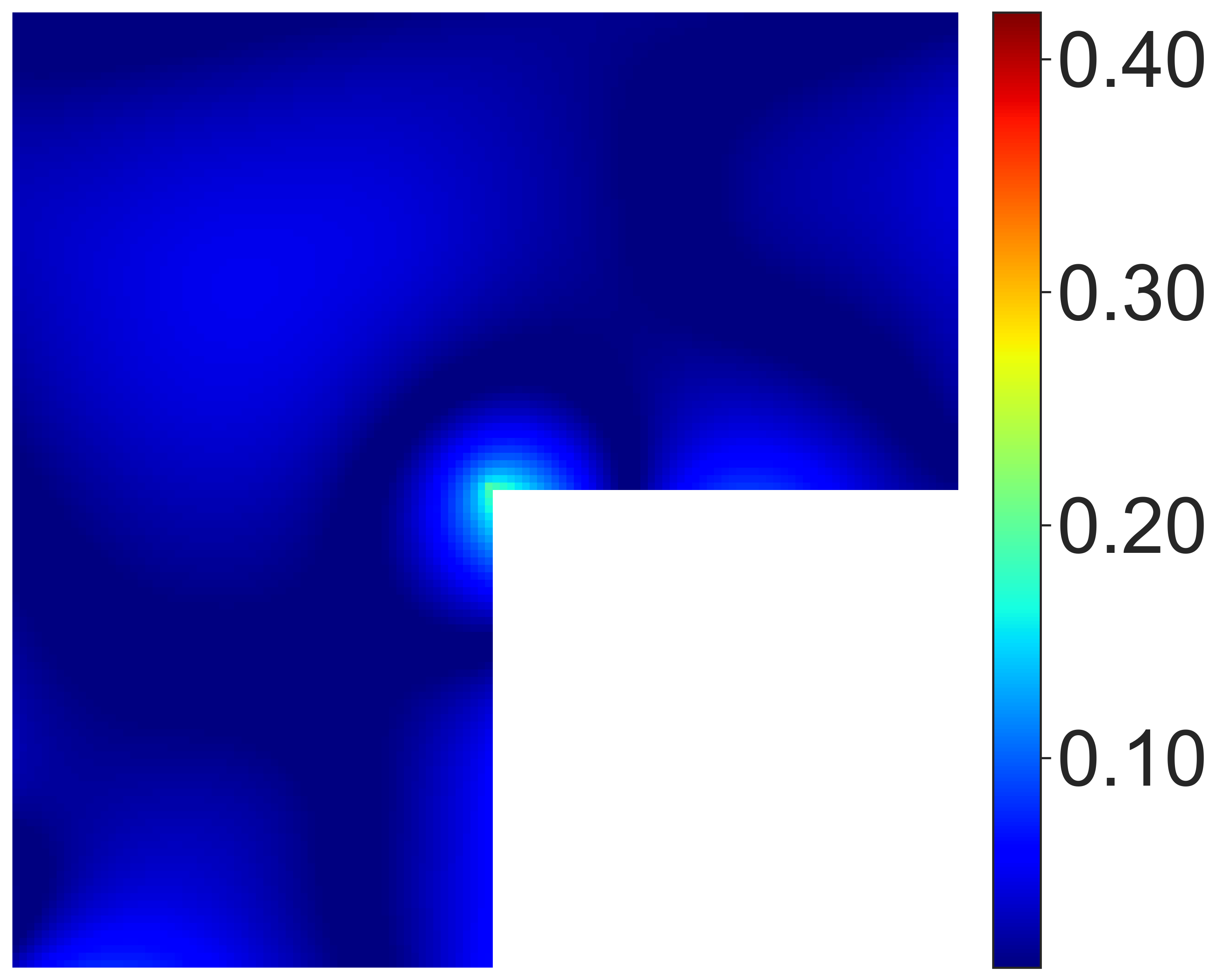}\\
         (a) Exact solution & (b) PINN  & (c) DRM  & (d) VPINN 
    \end{tabular}
    \caption{Approximate solutions in the top row and pointwise errors $|u_{\theta}(x)- u(x)|$ in the bottom row for Example \ref{example:Ldomain}.}
    \label{fig:2dLshape}
\end{figure}

In the next example, we show the performance of the method WAN. We compare its performance with all other methods listed in this paper.
\begin{example}[\bf WAN in 2D]\label{example:weak}
    We consider the Poisson problem with a weak solution. The domain $\Omega=(0,1)^2,$ with the coefficient matrix $\mathcal{A} = \mathbb{I}_{2 \times 2}$, $\bm{\beta} = [0 \; 0]^T,$ $c=0$ the source $f=2$, boundary data $g(x,y)=\min\{x^2,(1-x)^2\}$. This PDE has no strong solution\cite{zang2020weak}, but a weak solution exists: $y=\min\{x^2,(1-x)^2\}$.
\end{example}
For WAN, we employ AdaGrad\cite{duchi2011adaptive} to optimize both the solution network and the adversarial network. The learning rate is fixed to be $0.015$ for solution network and $0.04$ for test function. In each outer loop, solution network is trained by 3 iterations, and adversarial network is trained by 1 iteration. There are 5000 outer iterations in total. To compare the WAN with other methods, e.g., PINN, VPINN, and DRM we take the same training setting as in Example \ref{example:Ldomain}. For VPINN, we take 1264 nodes with 2574 elements. For PINN and DRM, we take 5000 points in the interior domain and 1000 points along the boundary.  

The results are displayed in Figure \ref{fig:weak}. The $L^2$ relative errors for PINN, DRM, VPINN, and WAN are 53.80\%, 37\%, 46\%, and 6.42\%, respectively. Among all these four methods, WAN yields the most accurate solution. DRM and VPINN is able to learn approximately the shape of solution, but the error is still large and concentrated in the center. The solution from PINN has a relatively large error spreading over domain and cannot approximate the correct shape of solution. Hence we further verify by this example that weak formulation-based methods can be beneficial for approximating solutions with low regularity. 

\begin{figure}[hbt!]
    \centering
    \begin{tabular}{ccccc}
        \includegraphics[width=0.16\textwidth]{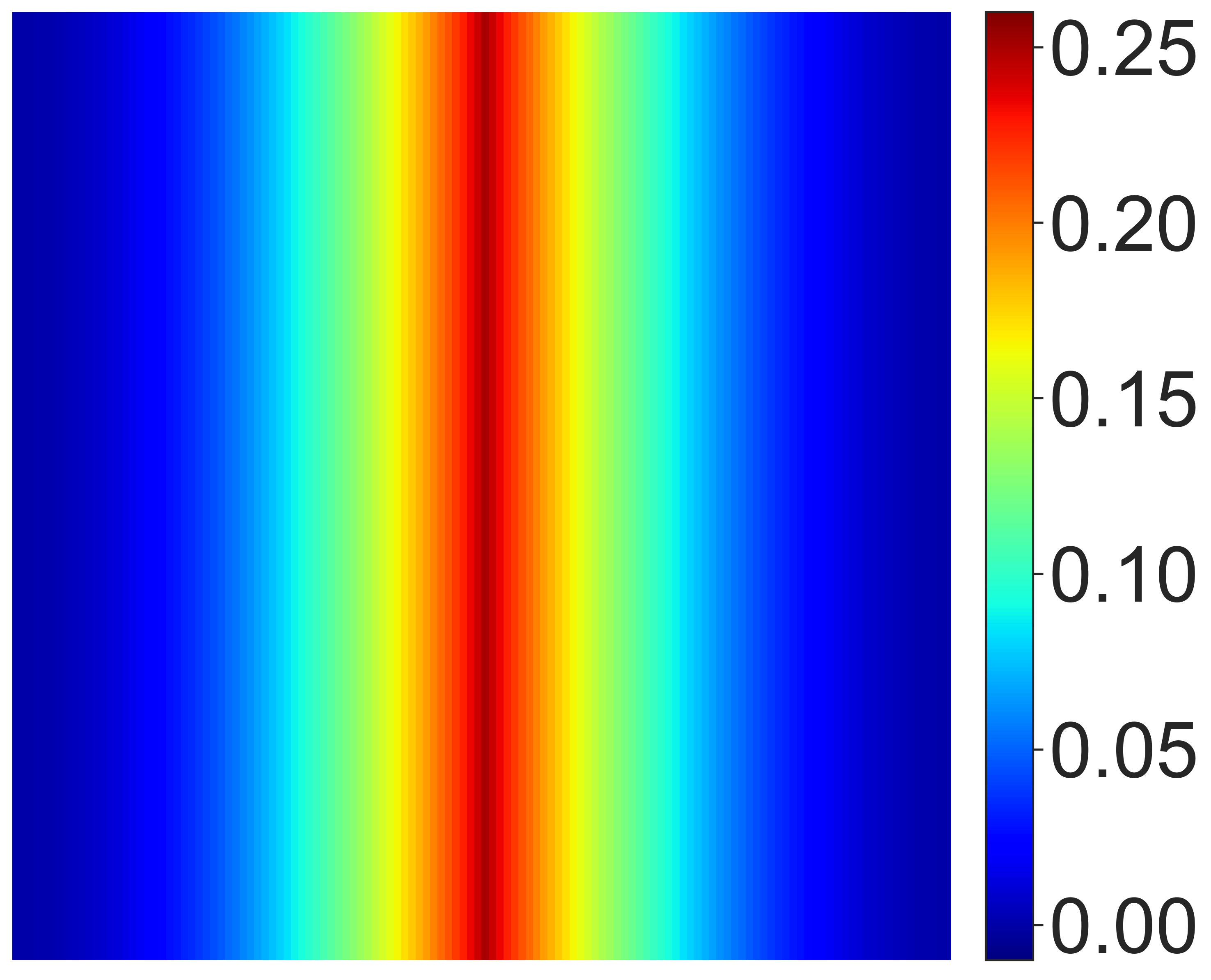} &
        \includegraphics[width=0.16\textwidth]{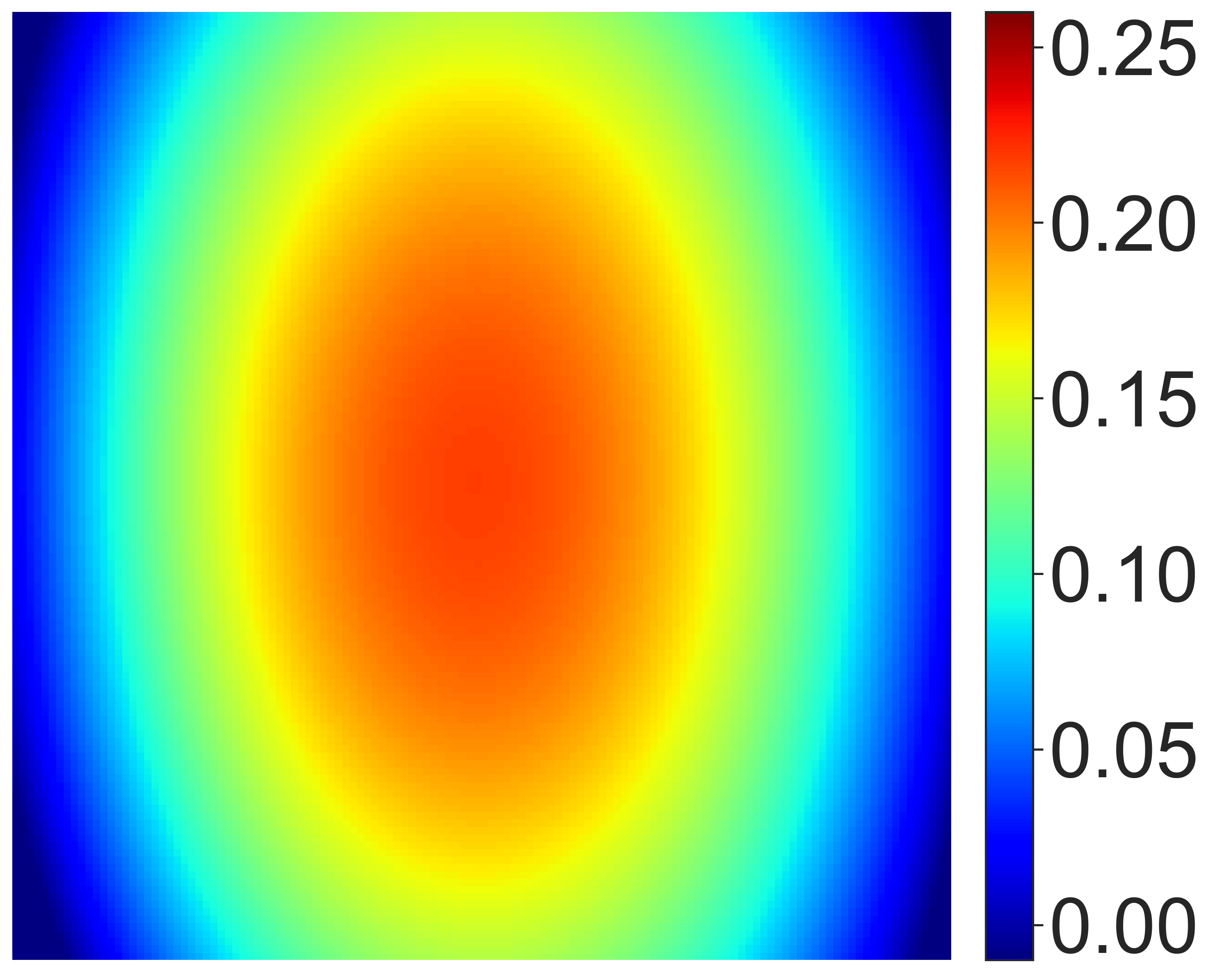} &
        \includegraphics[width=0.16\textwidth]{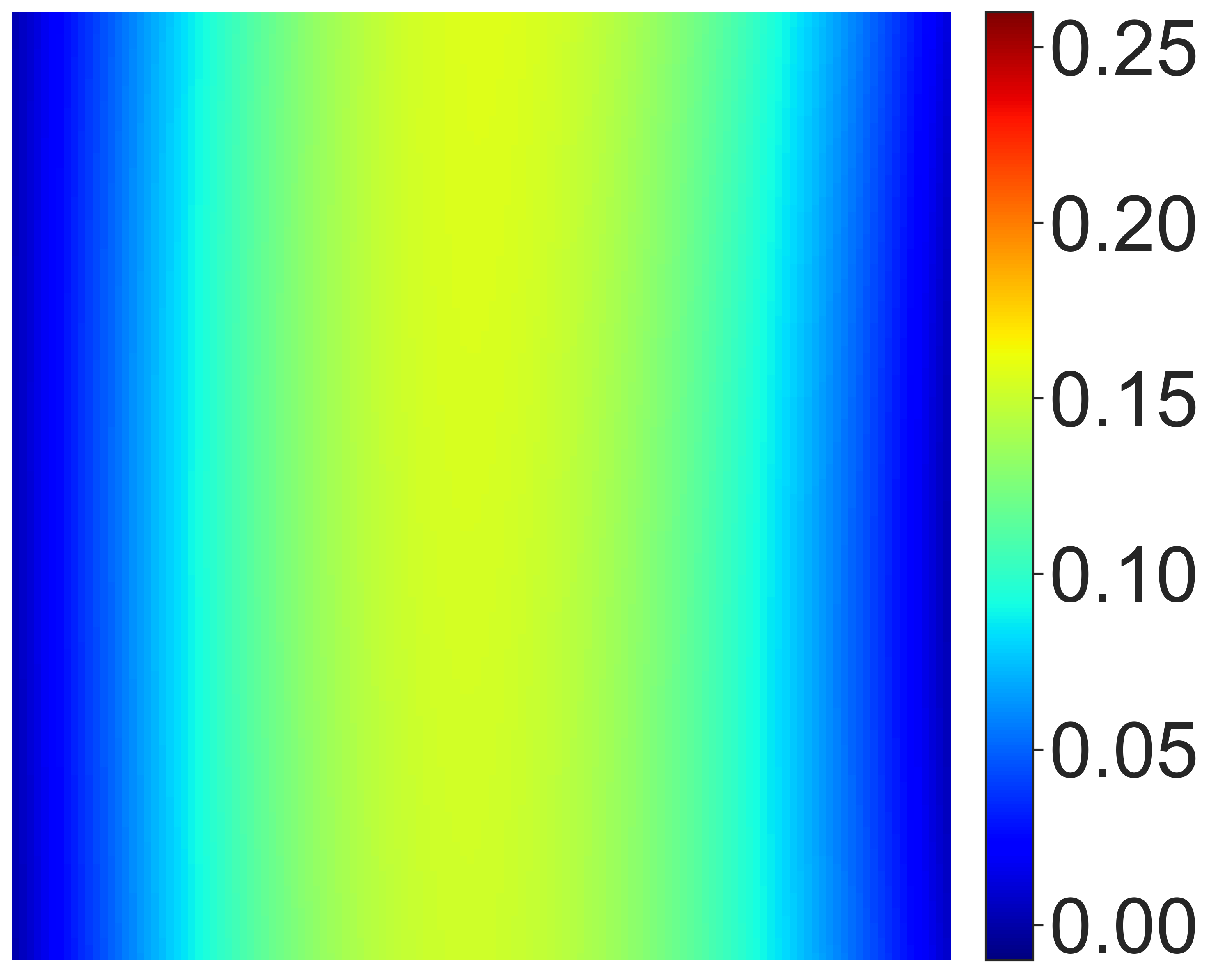} &
        \includegraphics[width=0.16\textwidth]{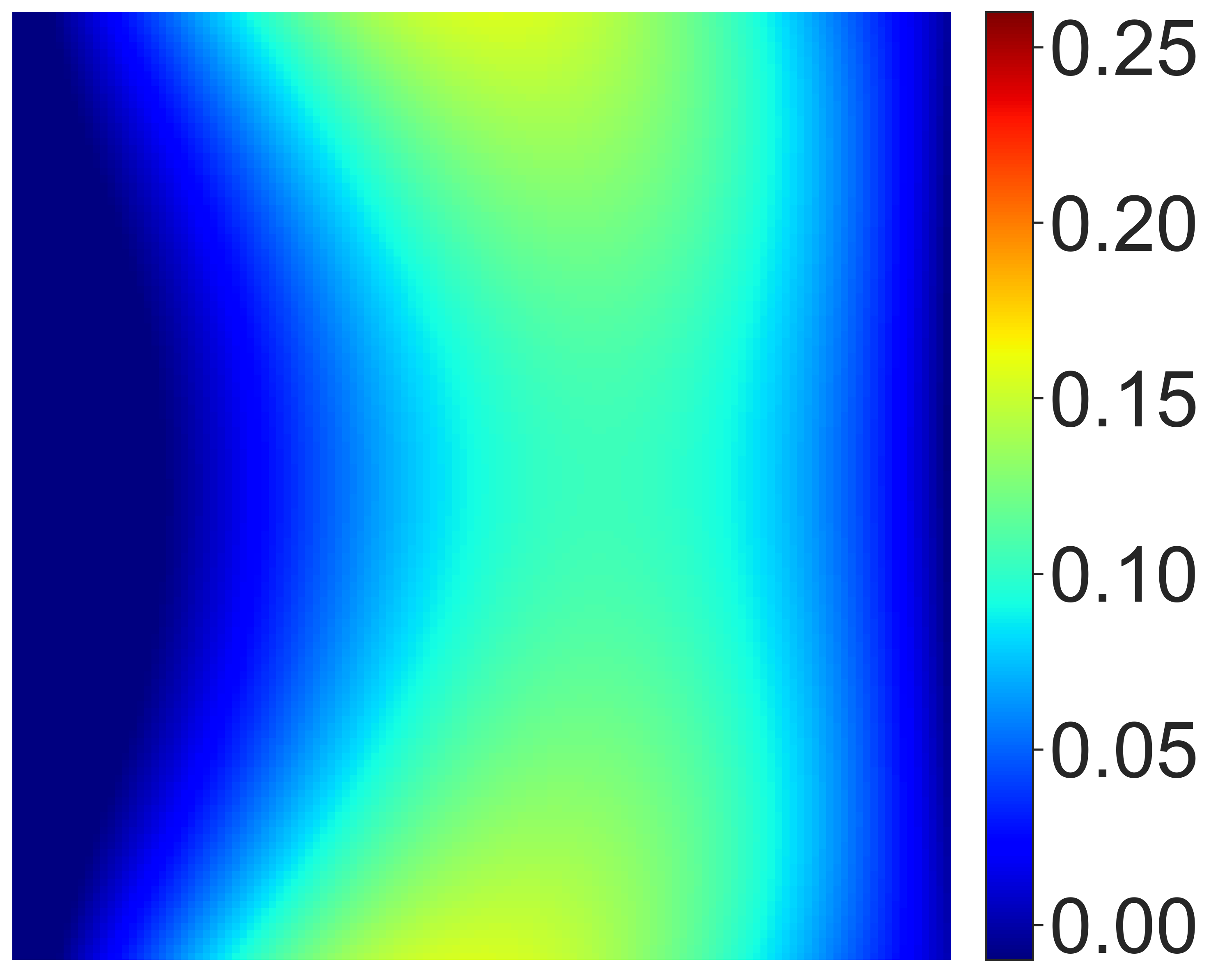} &
        \includegraphics[width=0.16\textwidth]{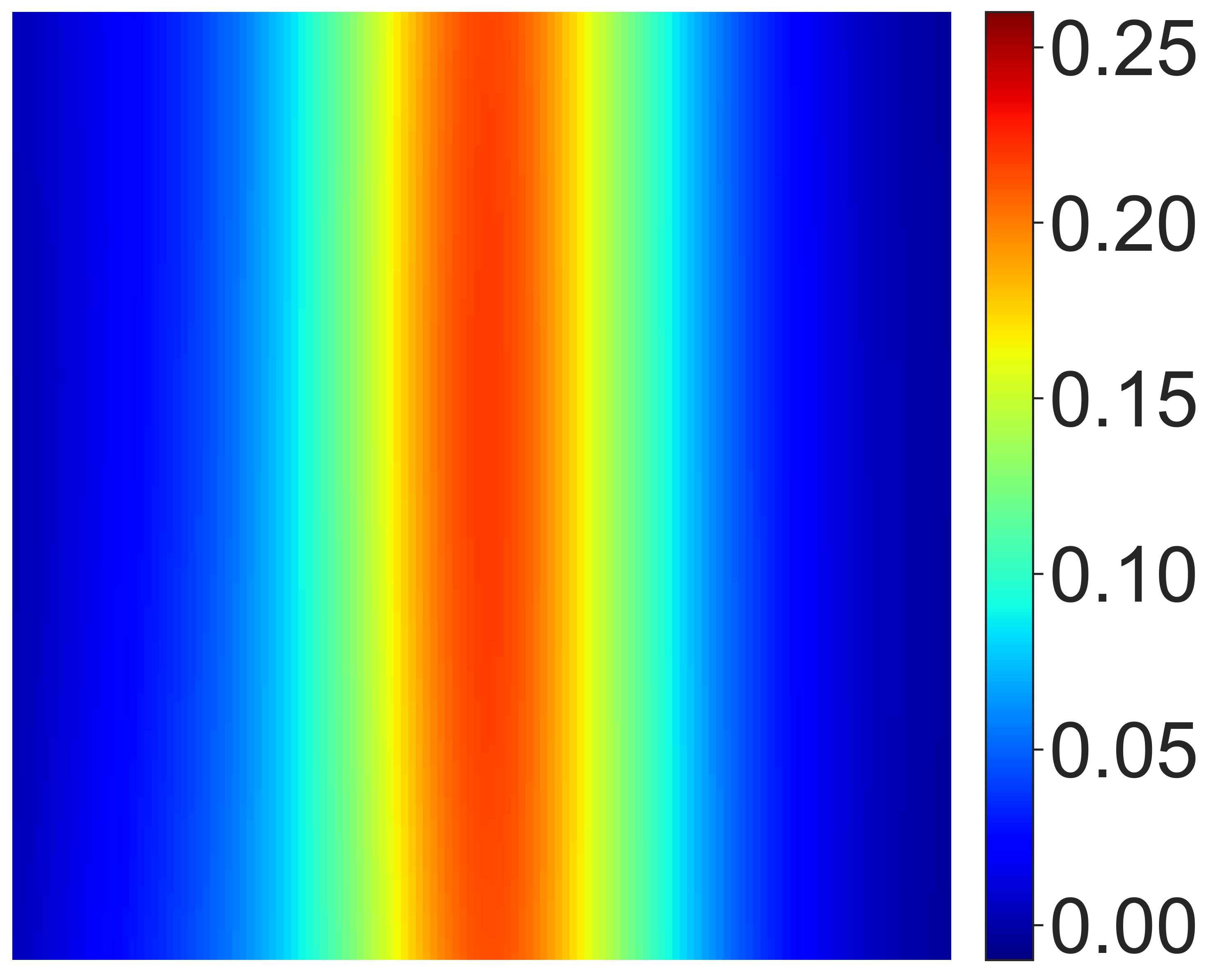} 
        \\
         &
         \includegraphics[width=0.16\textwidth]{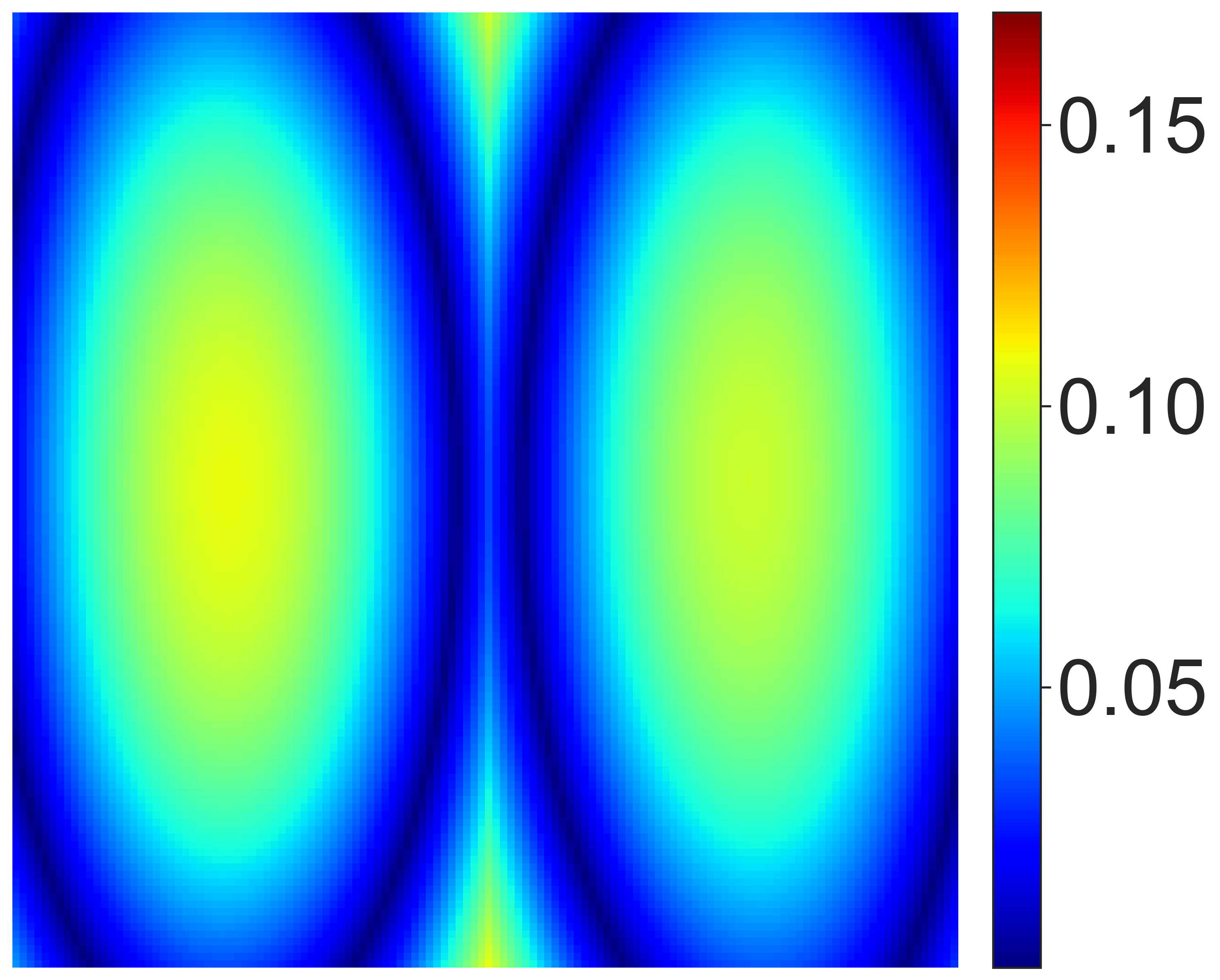} &
         \includegraphics[width=0.16\textwidth]{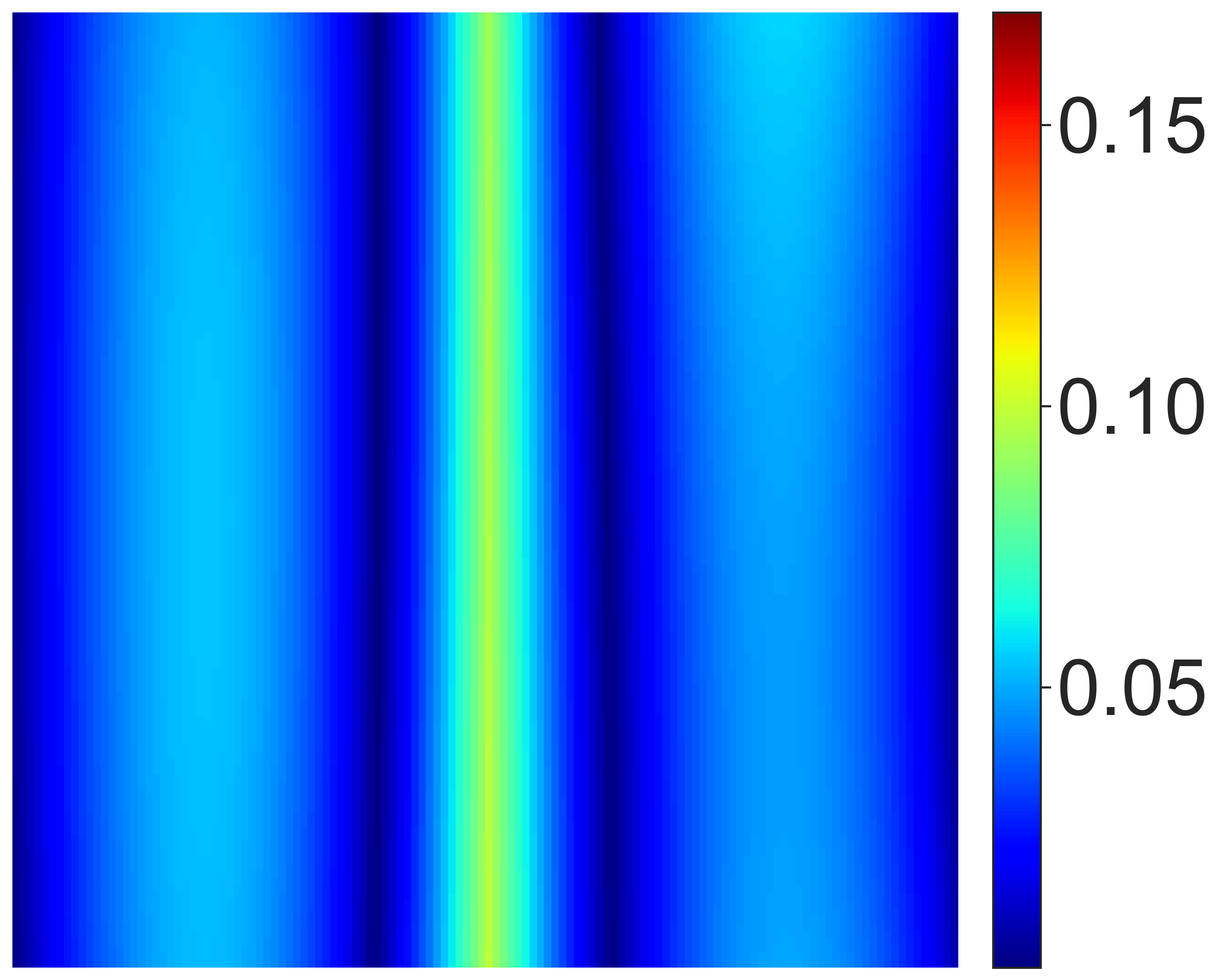} &
         \includegraphics[width=0.16\textwidth]{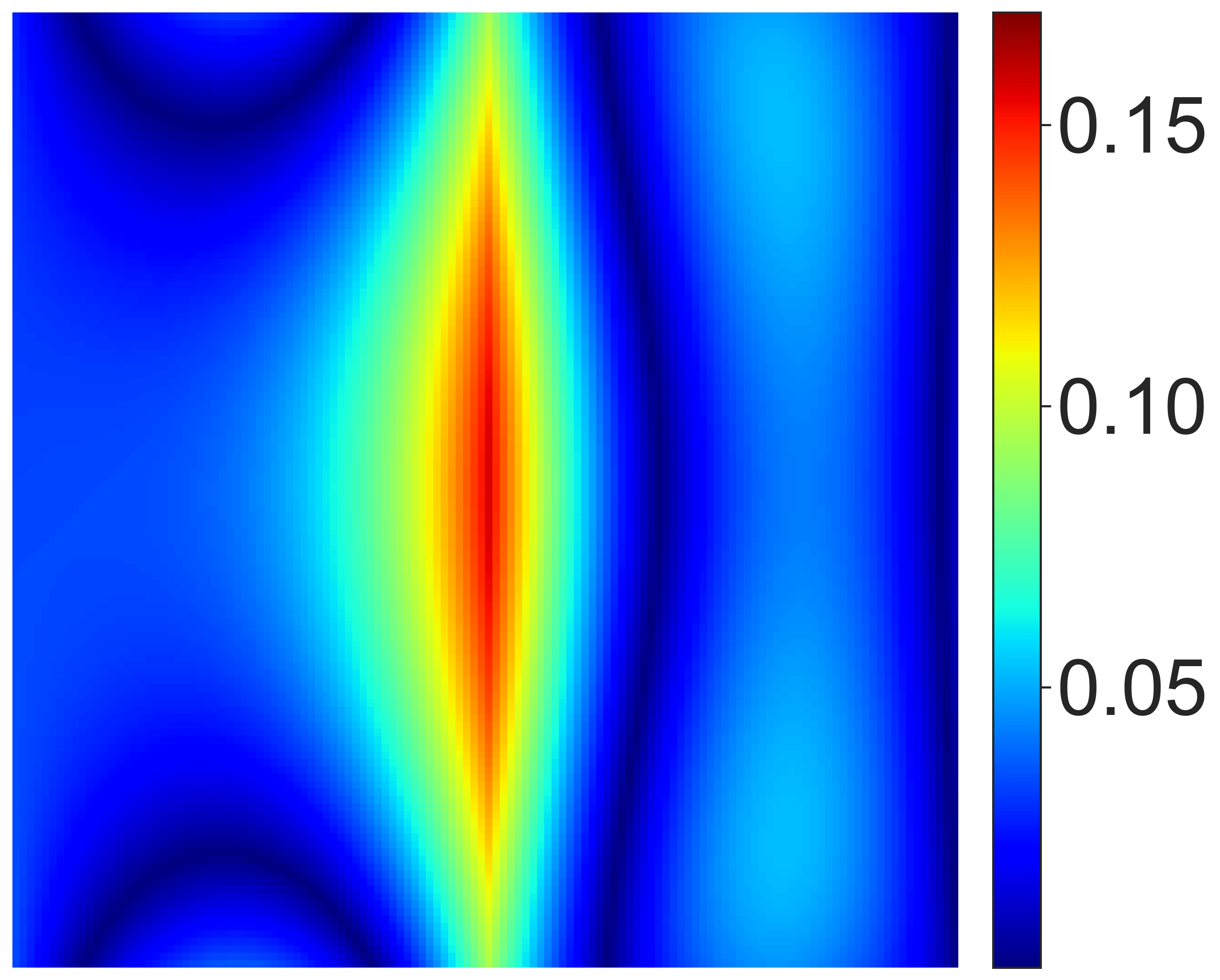} & 
         \includegraphics[width=0.16\textwidth]{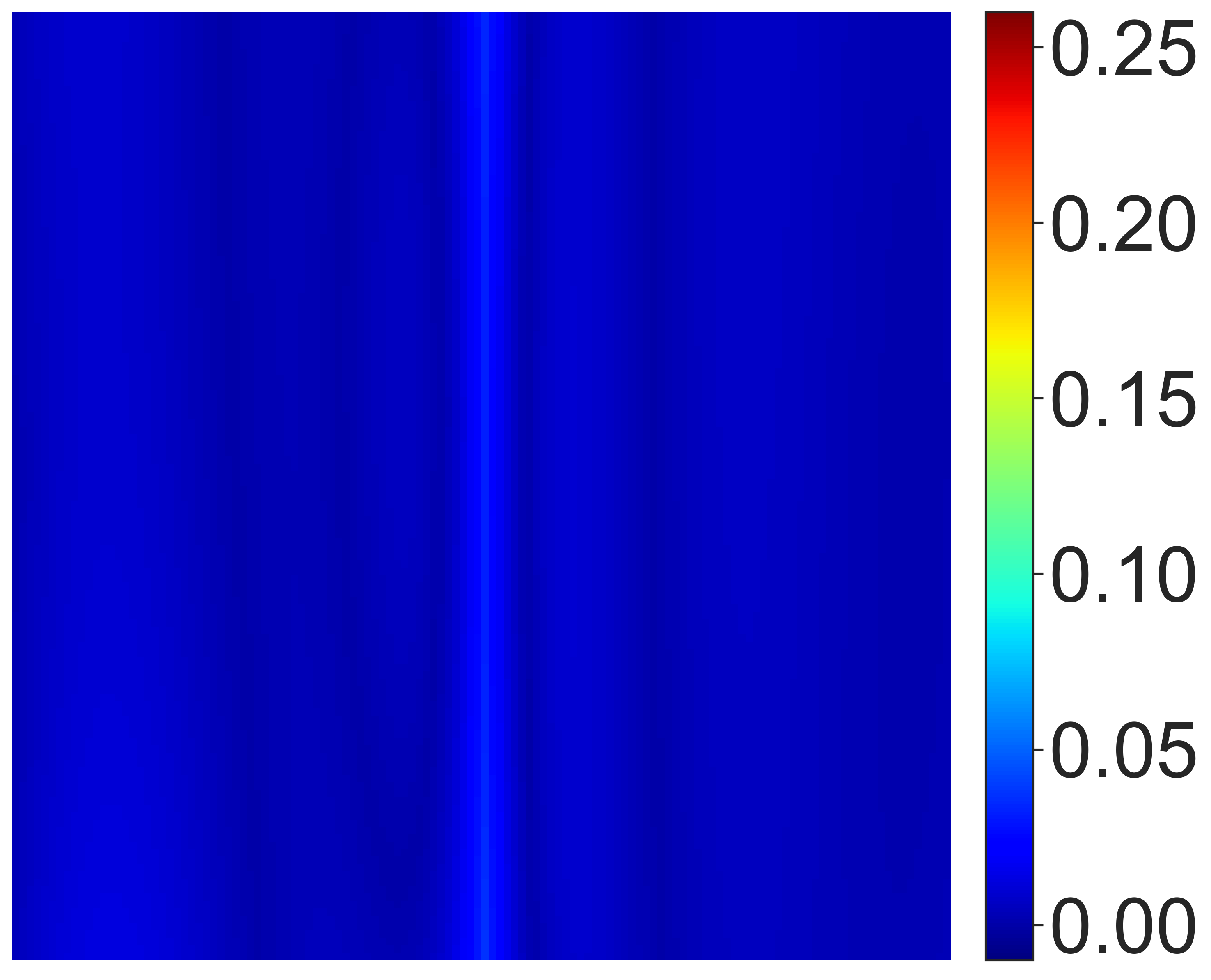}\\
(a) Exact solution & (b) PINN  & (c) DRM  & (d) VPINN & (e) WAN         
    \end{tabular}
    \caption{Approximate solutions in the top row and pointwise errors $|u_{\theta}(x)- u(x)|$ in the bottom row for Example \ref{example:weak}.}
    \label{fig:weak}
\end{figure}

\paragraph{Training instability of WAN} In particular, from experiments we observe that the training of WAN is very sensitive to the choice of hyperparameters, including a careful choice of iteration numbers. To demonstrate this phenomenon, three comparable tests are performed and the $L^2$ validation error is shown in Figure \ref{fig:wan_conv}. Data setting is same as in Example \ref{example:weak}, and maximum number of outer iteration is 5000. The following settings are taken differently:
\begin{enumerate}
    \item Test 1: Optimize by AdaGrad. The fixed learning rate is 0.015 for $u_\theta$, and 0.04 for $\phi_\eta$. In each outer iteration, $u_\theta$ is trained for 3 iterations, and $\phi_\eta$ for 1 iteration. This setting is adopted from the reference paper \cite{zang2020weak}.
    \item Test 2: Optimize by Adam, with fixed learning rate $10^{-4}$ for solution $u_\theta$ and $10^{-3}$ for test function $\phi_\eta$. In each outer iteration, $u_\theta$ is trained for 3 iterations, and $\phi_\eta$ for 1 iteration. 
    \item Test 3: Optimize by Adam. The fixed learning rate is $10^{-4}$ for $u_\theta$ and $10^{-3}$ for $\phi_\eta$. In each outer iteration, $u_\theta$ is trained for 6 iterations, and $\phi_\eta$ for 2 iteration.
\end{enumerate}

\begin{figure}[hbt!]
    \centering
    \includegraphics[width=0.5\textwidth]{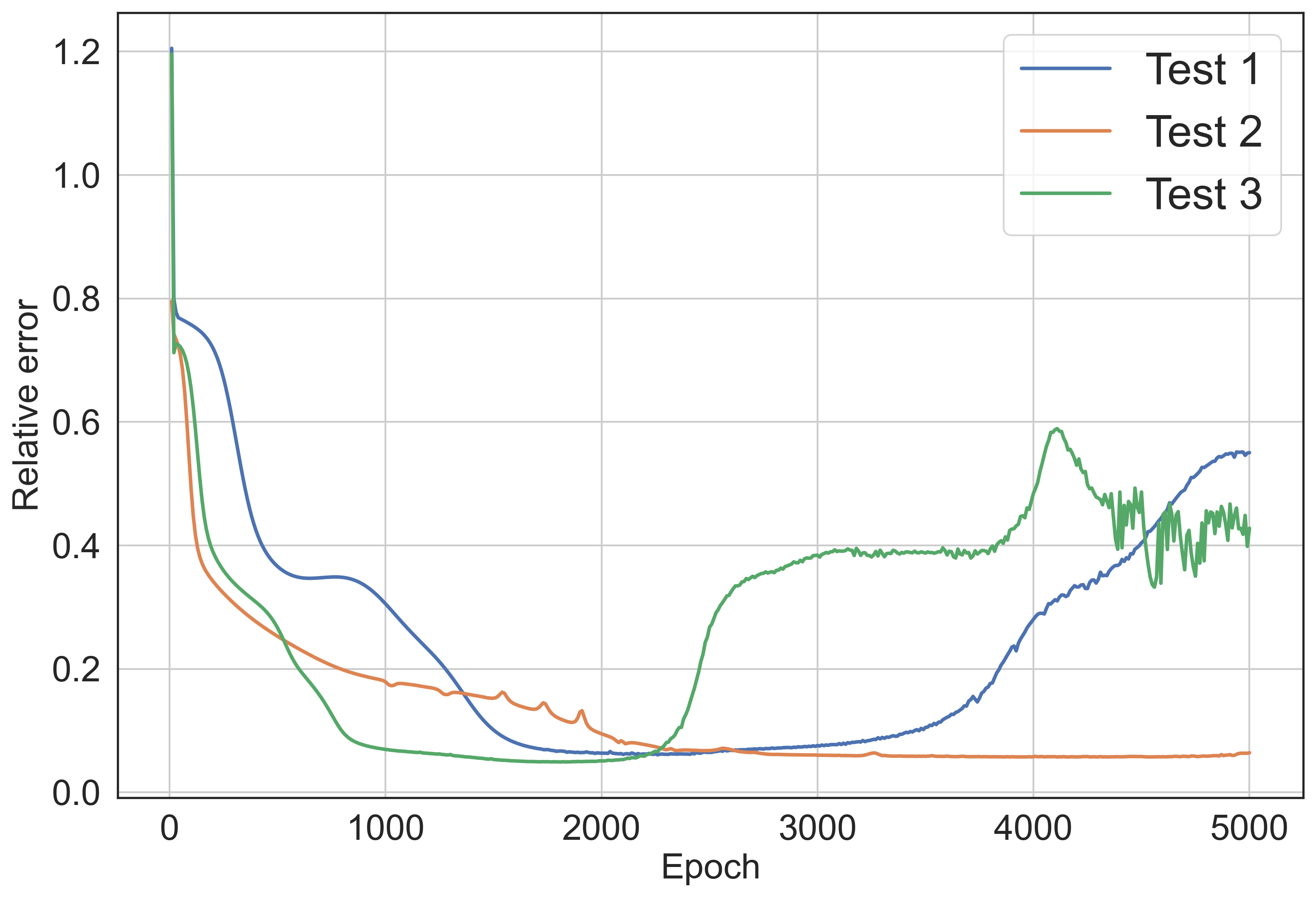}
    \caption{The validation error history of for three WAN tests. }
    \label{fig:wan_conv}
\end{figure}

In test 1 and 3, the validation error of approximated solution $u_\theta$ first decreases and then increase. Test 3 increases number of iterations in each outer iteration(epoch), and the worsening of solution becomes faster. This means that a proper early stopping rule is required in order to obtain an accurate solution before it goes worsening. This phenomena can be explained by the ill-posedness of the optimization problem of WAN. For such problems one typically need to do early stopping and adopt a carefully designed stepsize rule to yield reasonable result. In test 2, although the solution is still worsening, but the process is much slower. 

Concluding from the results and the observation above, WAN is able to approximate weaker solutions and can beat all other NN-based methods in some examples, but the training difficulty prevents it from efficient application. Stablizing tricks for WAN will hugely improve the practice.



\subsection{High dimension examples}
\begin{example}[\bf PINN]\label{example:pinn4d}
We consider the general elliptic equation \eqref{eqn:Poisson} in the four dimensional domain $\Omega=(0,1)^4$, with the coefficient matrix $\mathcal{A} = \mathbb{I}_{4\times 4}$, $\bm{\beta} = [0 \; 0\; 0 \; 0]^T,$ $c=0$, the source: $$f = 4\pi^2 \sin(\pi x_1) \sin(\pi x_2) \sin(\pi x_3) \sin(\pi x_4),$$ boundary condition $g=0$ and known exact solution $u = \sin(\pi x_1) \sin(\pi x_2) \sin(\pi x_3) \sin(\pi x_4).$
\end{example}

 We use a DNN architecture with 4 hidden layers and 80 neurons in each layer. When formulating the empirical loss $\widehat{\mathcal{L}}(u_{\theta})$, we take $n=20000$ points i.i.d. from $U(\Omega)$ for the PDE residual, and $m=5000$ points i.i.d. from $U(\partial \Omega)$, for boundary residual and boundary weight $\alpha = 200$. We minimize the loss $\widehat{\mathcal{L}}(u_{\theta})$ using ADAM provided by the PyTorch library \texttt{torch.optim} (version 2.0.0). We use 20000 ADAM iterations to minimize the loss function, the learning rate = $1e-4$. The exact solution and computed neural network solution and pointwise errors are shown in Figure \ref{fig:4D_Poisson} and loss history has been shown in Figure \ref{fig:4D_loss_hist}.
 \begin{figure}[h]
    \centering
    \setlength{\tabcolsep}{0pt}
    \begin{tabular}{ccc}
            \includegraphics[width=0.24\textwidth]{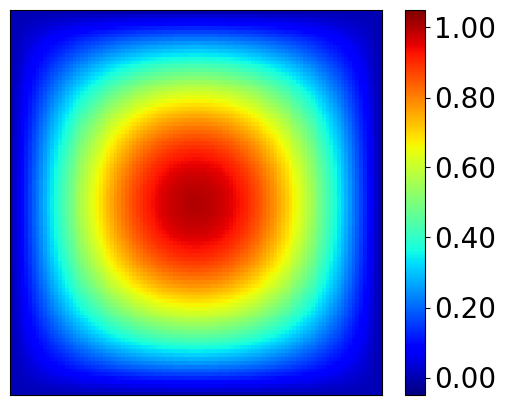}
            &\includegraphics[width=0.24\textwidth]{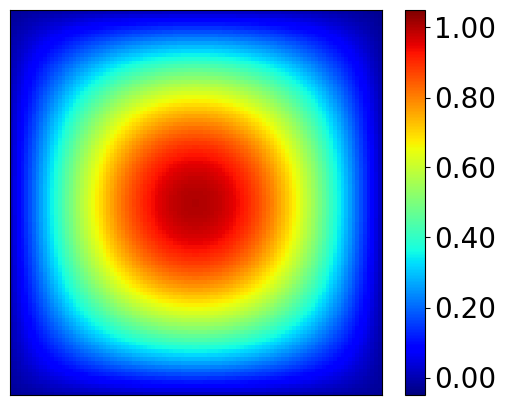}
            &\includegraphics[width=0.24\textwidth]{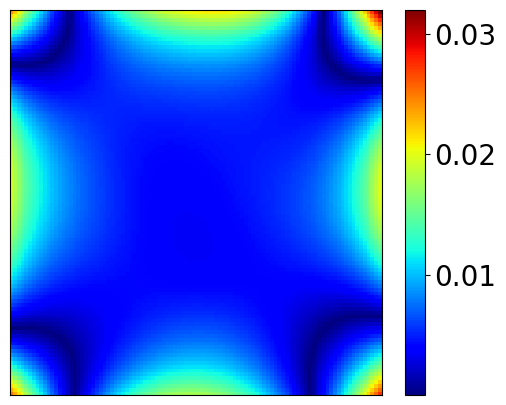}\\
            (a)  & (b) & (c) 
    \end{tabular}
    \caption{(a) Exact solution($u$), cross section at $x_3=x_4=0.5$ (b) Neural network approximate solution($u_{\theta}$), cross section at $x_3=x_4=0.5$ (c) pointwise error $|u_{\theta}(x)- u(x)|$, for Example \ref{example:pinn4d}}
    \label{fig:4D_Poisson}
\end{figure}

\begin{figure}
    \centering
    \includegraphics[width=0.50\textwidth]{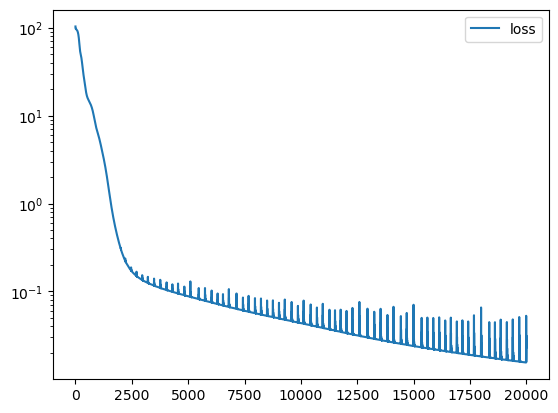}
    \caption{Loss history($\hat{\mathcal{L}}(u_{\theta})$), for Example \ref{example:pinn4d}}
    \label{fig:4D_loss_hist}
\end{figure}

\begin{example}\label{example:6dpoisson}
    We consider the general elliptic equation \eqref{eqn:Poisson} in a higher dimension $d=6$. The domain $\Omega=[0,1]^6$, coefficient matrix $\mathcal{A} = \mathbb{I}_{6\times 6}$, $\bm{\beta} = [0 \; 0\; 0 \; 0\; 0\; 0]^T,$ $c=0$, the source $f = 0,$ $g(x_1,\ldots,x_6)=x_1x_2+x_3x_4+x_5x_6$. The known exact solution $u = x_1x_2+x_3x_4+x_5x_6.$
\end{example}

We test and compare the performance of PINN and DRM in this example. We employ the same neural network architecture as in Example \ref{example:pinn4d}. The domain data of size $n=10000$ are uniformly sampled from domain, and boundary data of size $m=2000$ uniformly from boundary. For both methods, the boundary weight is taken to be $100$, selected by trial and error. We use 10000 ADAM iterations to optimze the NN parameters, with initial learning rate $10^{-3}$, multiplied by 0.1 at 5000-th and 7000-th iteration. For WAN, the boundary weight is $10^{6}$.  We use AdaGrad\cite{duchi2011adaptive} with learning rate 0.015 and 0.04 for solution NN and test function, respectively. We record the dynamics of loss value and early-stop the training at 2800-th outer iteration, before the loss value starts to rise. In each outer iteration, solution NN is trained for 3 iterations and test function for 1 iteration. The figures are displayed in \ref{fig:6dpoisson}. The $L^2$ relative errors for PINN, DRM and WAN are 
$0.55\%, 1.45\%$ and $6.59\%$, respectively.

\begin{figure}[hbt!]
    \centering
    \begin{tabular}{cccc}
        \includegraphics[width=0.19\textwidth]{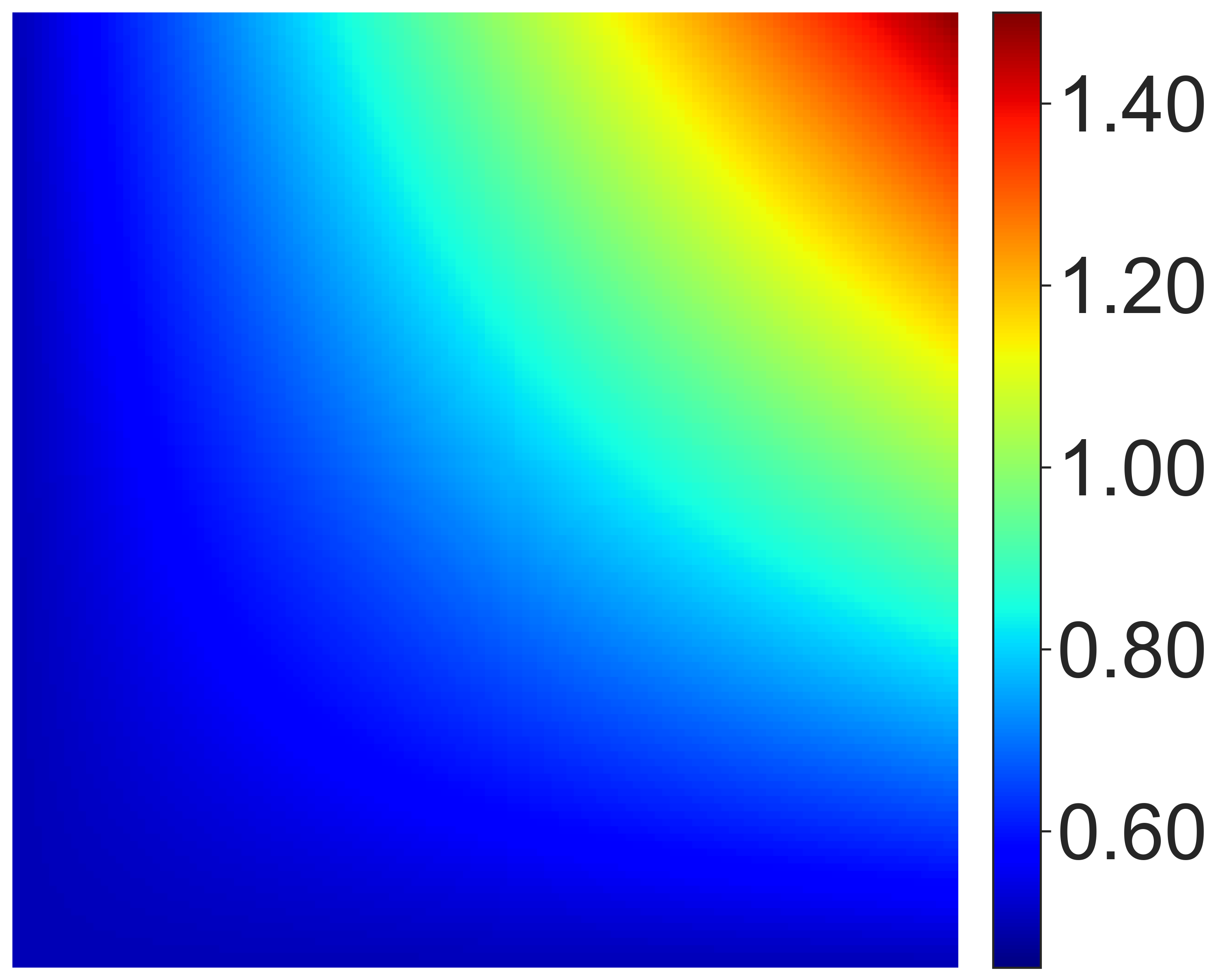} &
        \includegraphics[width=0.19\textwidth]{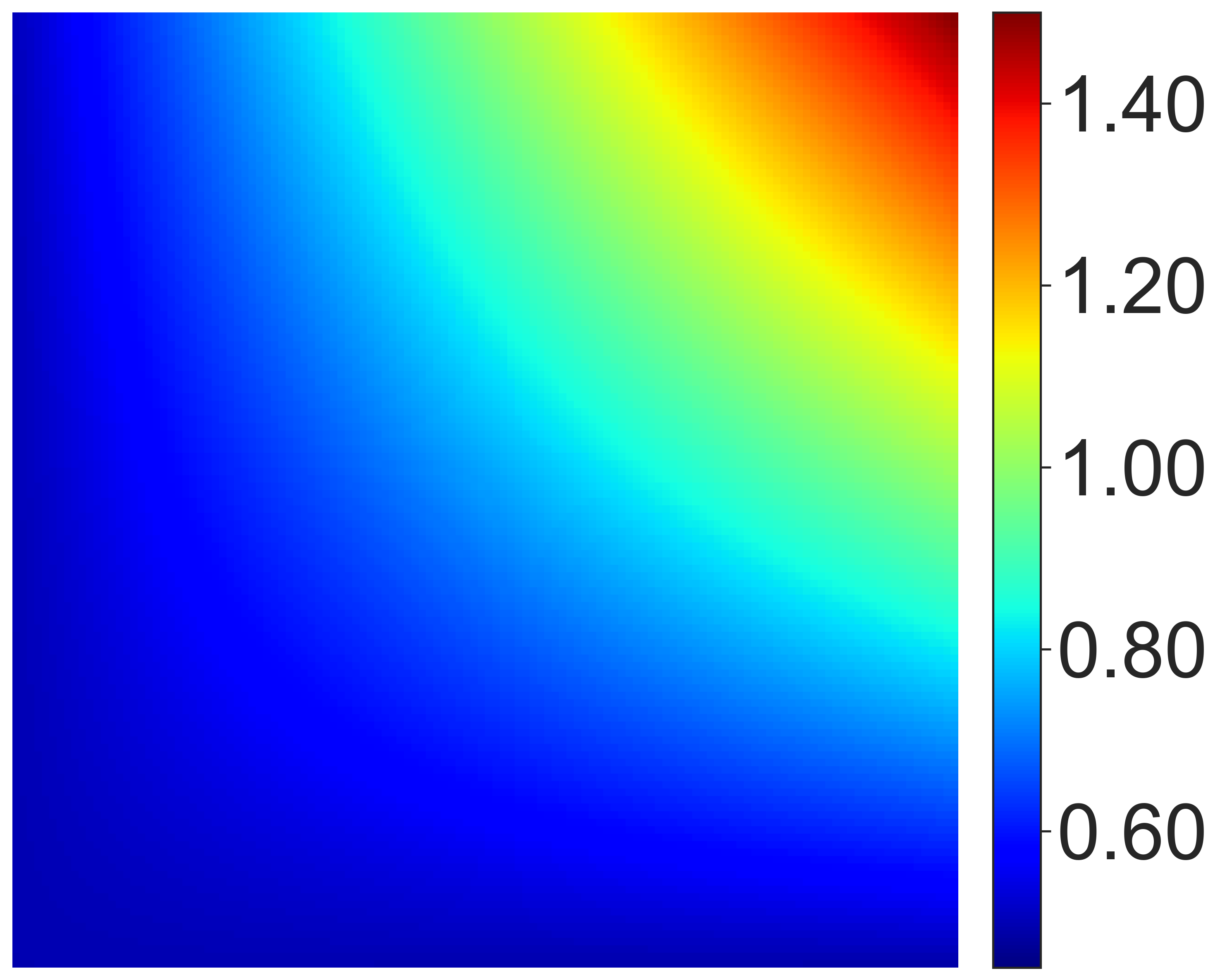} &
        \includegraphics[width=0.19\textwidth]{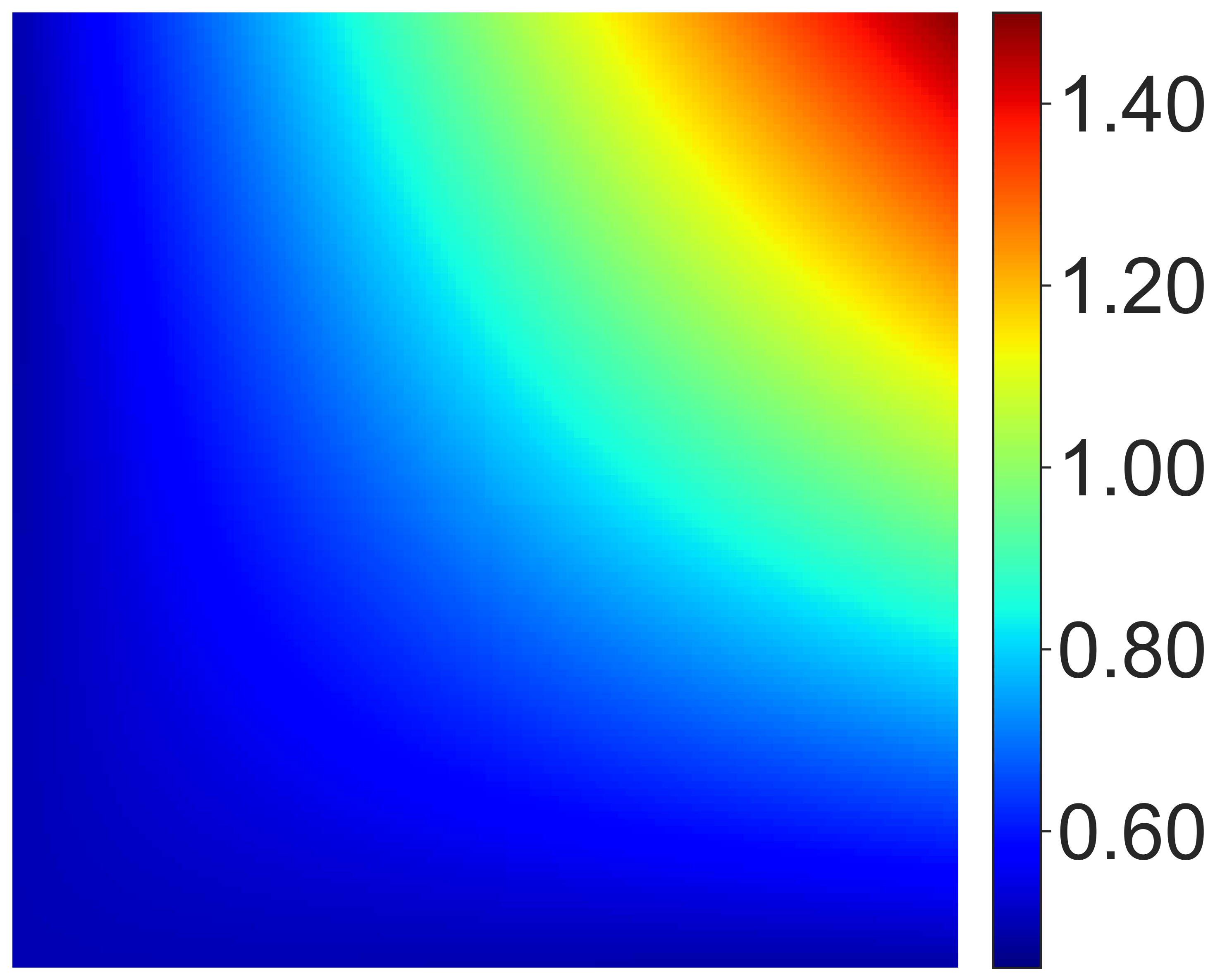} & 
        \includegraphics[width=0.19\textwidth]{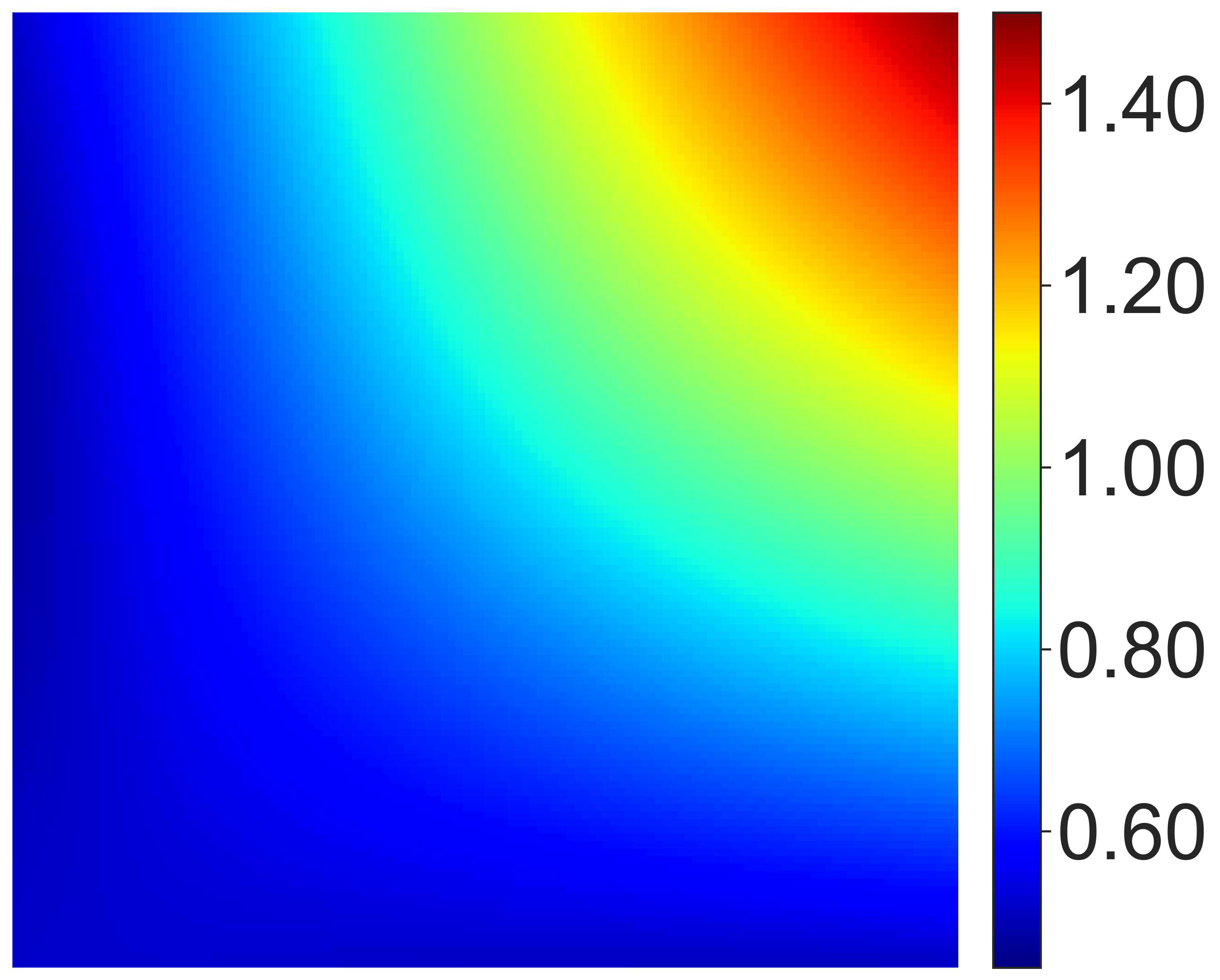} \\
         &
         \includegraphics[width=0.19\textwidth]{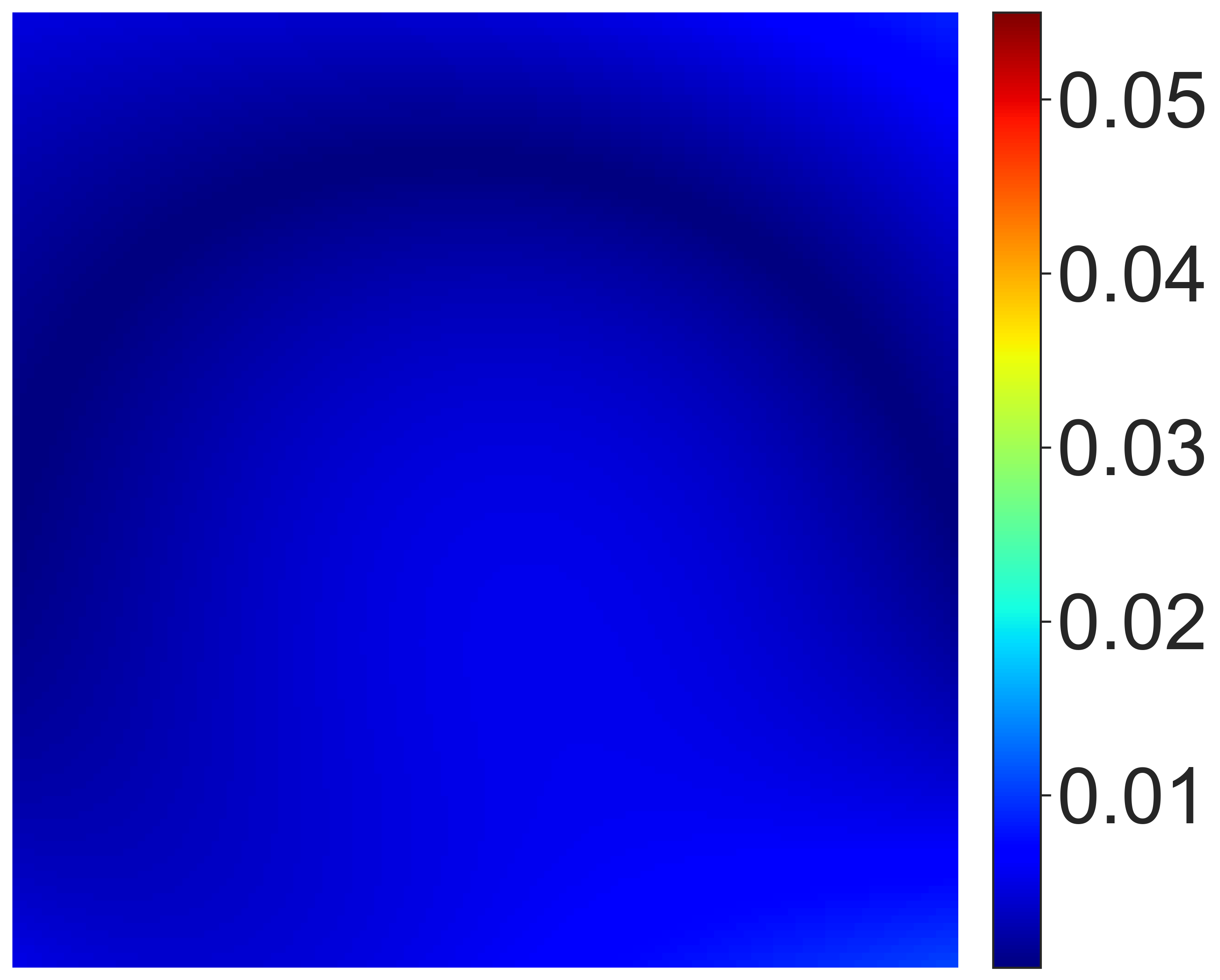} &
         \includegraphics[width=0.19\textwidth]{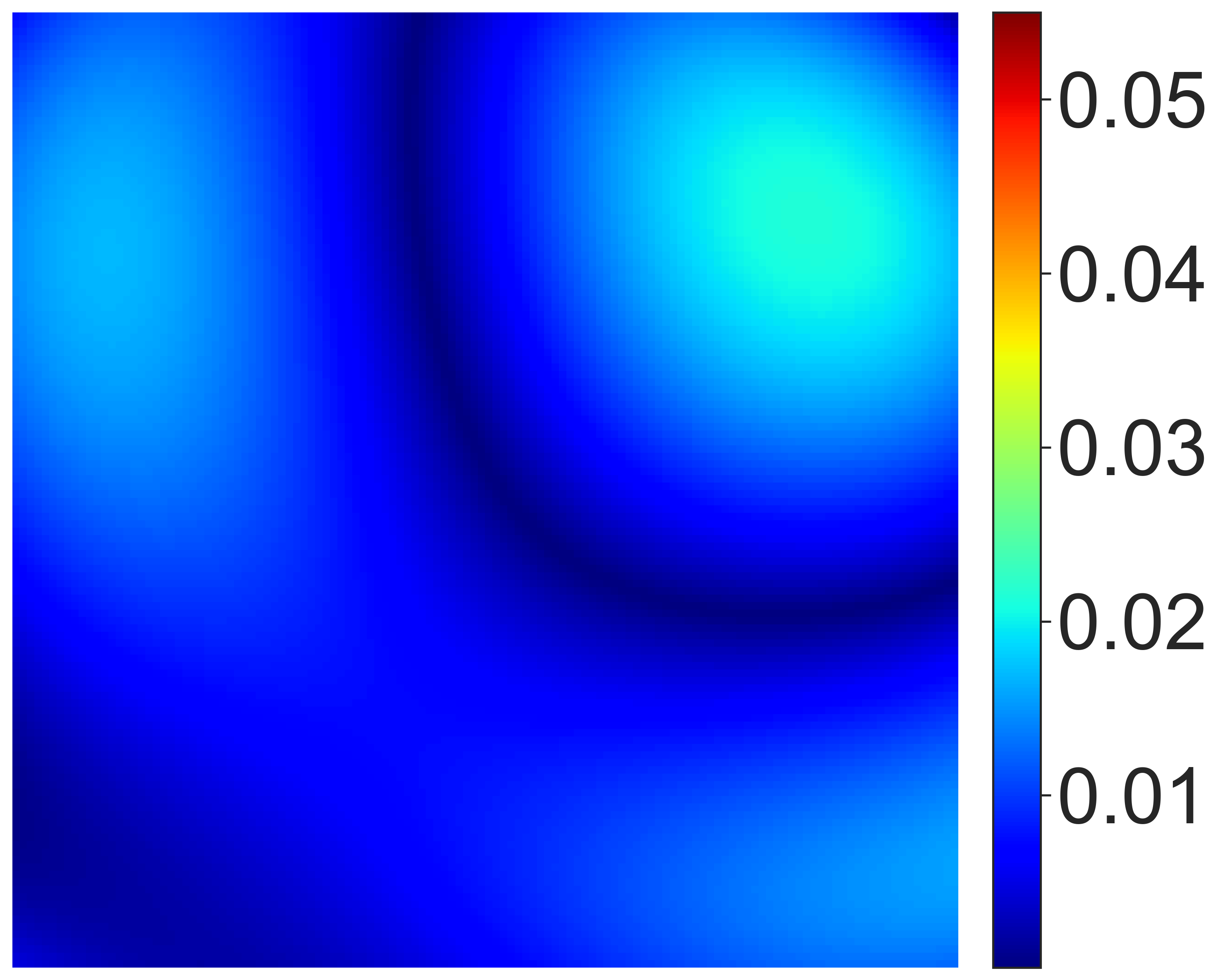} & 
         \includegraphics[width=0.19\textwidth]{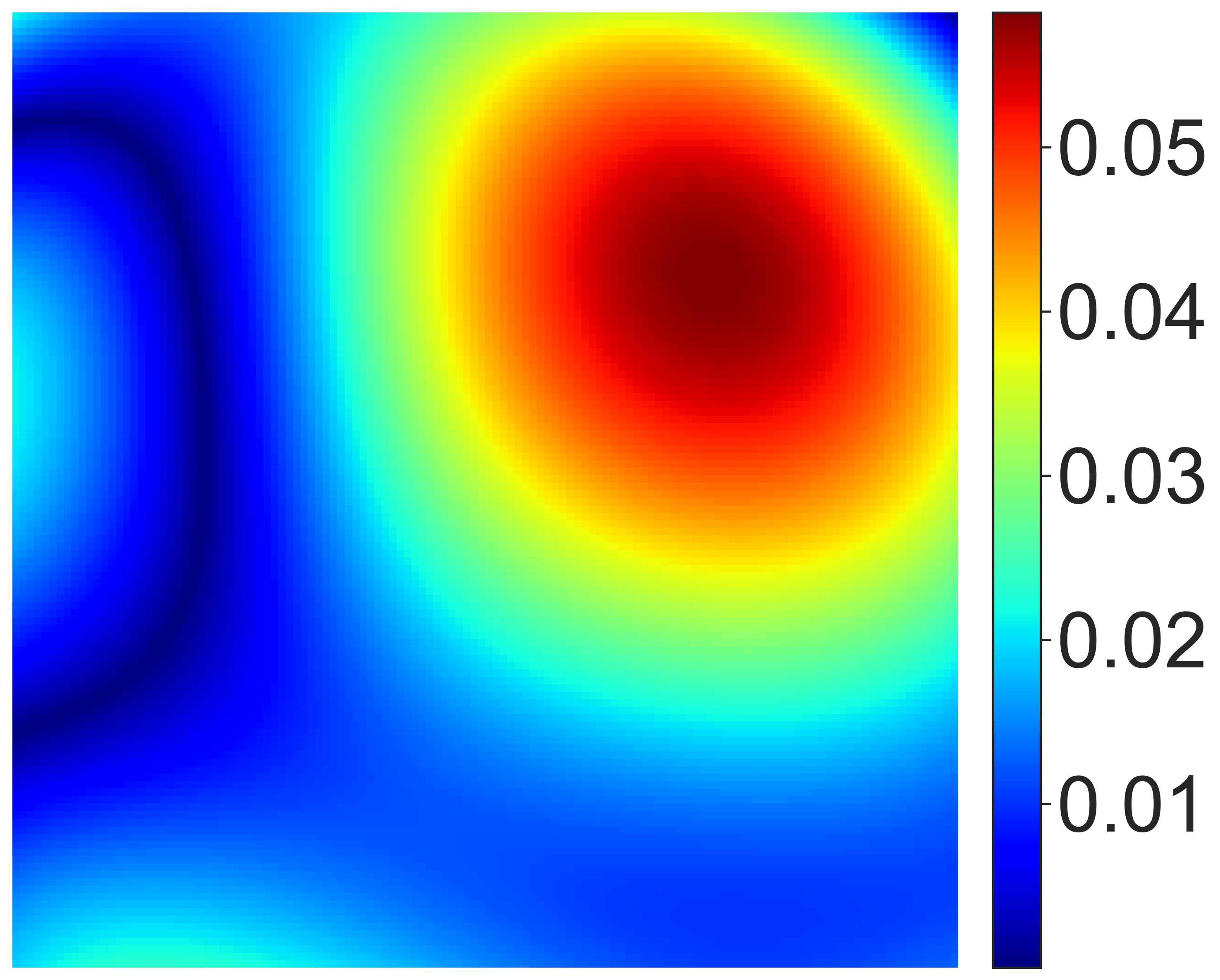} \\
         (a) Exact solution & (b) PINN  & (c) DRM & (d) WAN
    \end{tabular}
    \caption{Approximate solutions in the top row and pointwise errors $|u_{\theta}(x)- u(x)|$ in the bottom row for Example \ref{example:6dpoisson}.}
    \label{fig:6dpoisson}
\end{figure}

The training dynamics is shown in Fig \ref{fig:6dpoisson_dyn}, where the loss value and the validation error is recorded during training. The changing of loss value in PINN and DRM are oscillative at early state due to a large learning rate, but becomes stable after the learning rate being reduced. The loss value decreases until convergence, and validation error decreases in accordingly. The training process of WAN behaves in a similar way as in Example \ref{example:weak}: the solution becomes increasingly accurate at the early stage of training, but the error may start to increase drastically later. The figure indicates that the most accurate solution of WAN during training has about $1\%$ error, which is comparable to DRM. The WAN solution present in Figure \ref{fig:6dpoisson} is the final solution after whole training process, and the error is larger than other two methods.

\begin{figure}[hbt!]
    \centering
    \begin{tabular}{ccc}
        \includegraphics[width=0.28\textwidth]{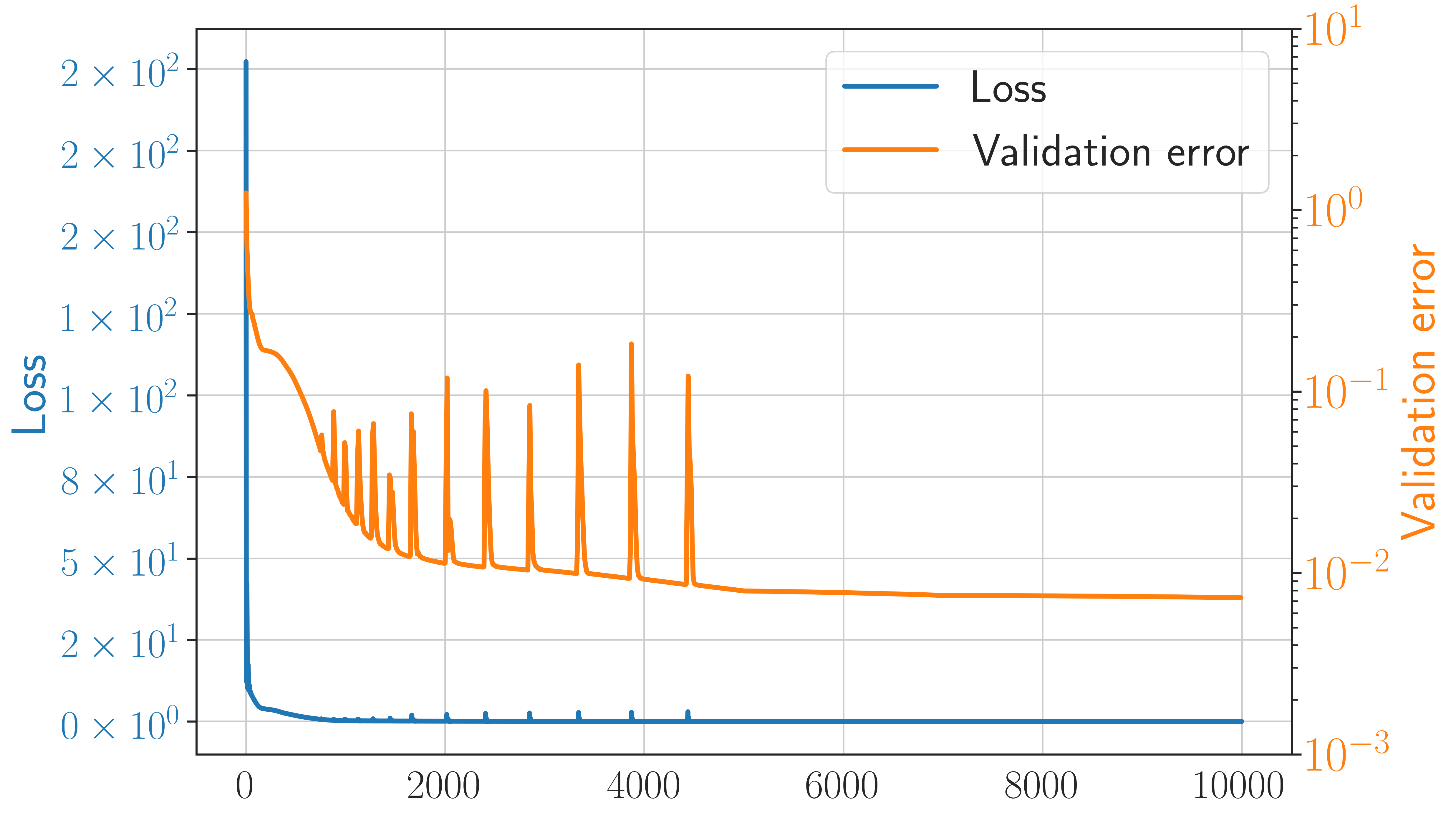} &
        \includegraphics[width=0.28\textwidth]{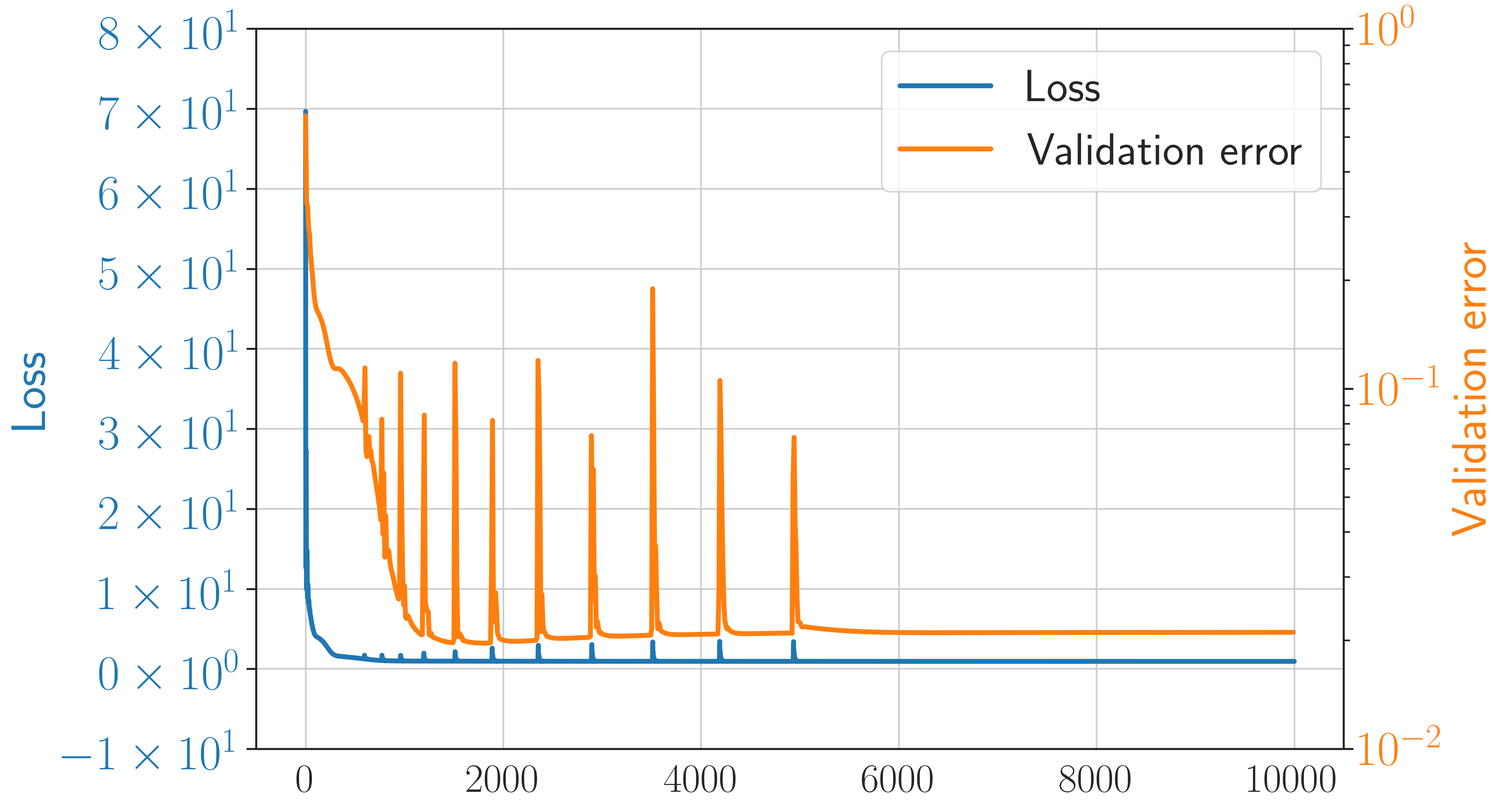} &
        \includegraphics[width=0.28\textwidth]{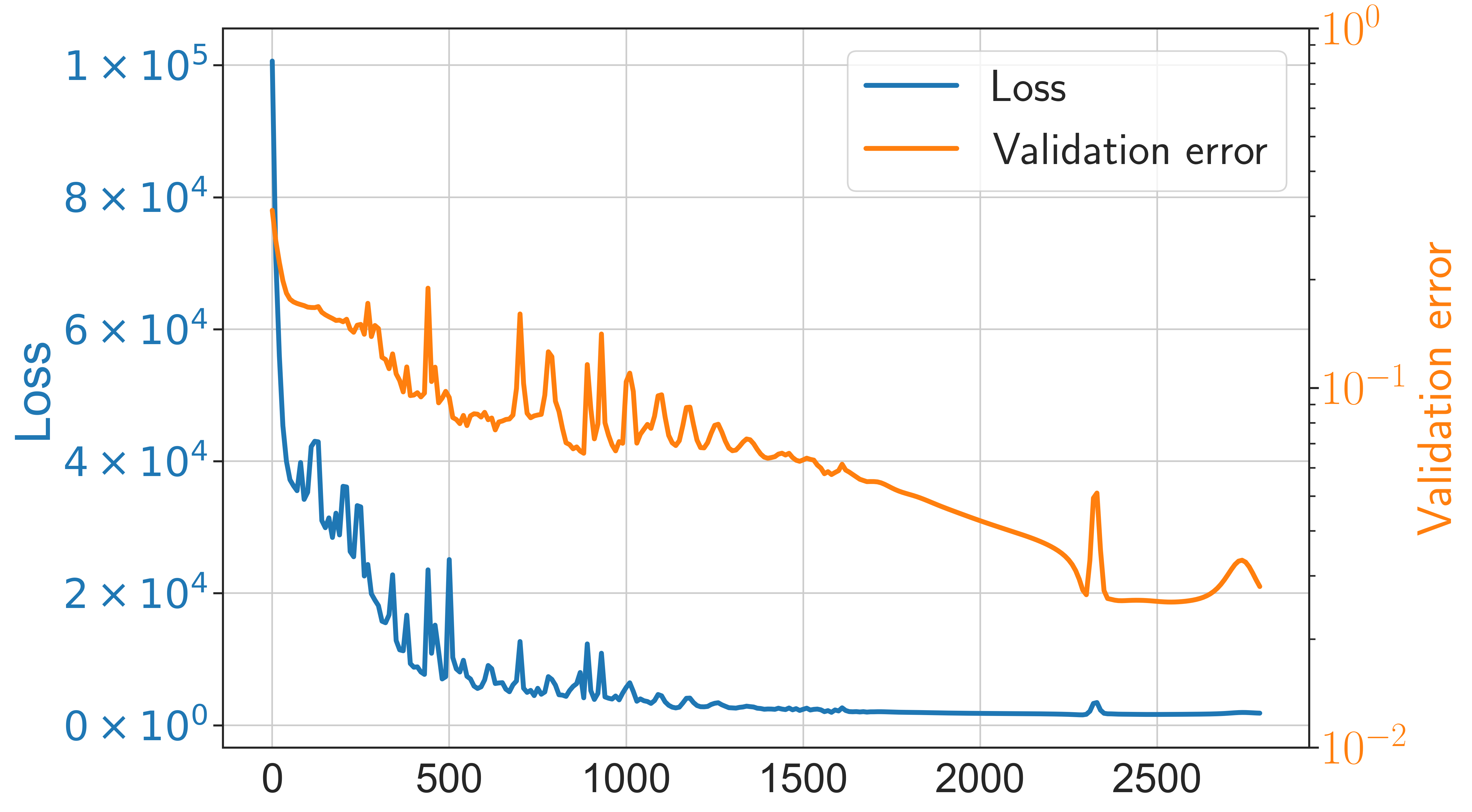} \\
          (a) PINN  & (b) DRM & (c) WAN
    \end{tabular}
    \caption{Training dynamics for Example \ref{example:6dpoisson}. Statistics in PINN and DRM are recorded after each 10 iterations; in WAN are recorded after each outer iteration.}
    \label{fig:6dpoisson_dyn}
\end{figure}

For the completeness of this mathematical review of DNN methods for partial differential equations, we present a detailed error analysis for PINN method in the next section. We briefly describe the error analysis for DRM, VPINN, and WAN in Remark \ref{rem:error_estm_DRM}, \ref{rem:error_estm_VPINN} \& \ref{rem:error_estm_WAN}, detailed error analysis can be found in \cite{JiaoLai:2021error,berrone2022variational,bertoluzza2023best} respectively. 

\section{Error analysis for PINNs method}\label{sec:error_analysis}
 For the simplicity of the error analysis for the PINN method we choose $A=\mathbb{I}\in \mathbb{R}^{d\times d},  \bm{\beta}=\mathbf{0}\in \mathbb{R}^d, c\equiv 0$ in \eqref{eqn:Poisson0}. Therefore our general second-order elliptic PDE \eqref{eqn:Poisson0} reduces to the following Poisson problem:
\begin{equation}\label{eqn:Poisson}
    \left\{\begin{aligned}
         -\Delta u &= f \quad \mbox{in }\Omega,\\
        u & = g\quad \mbox{on }\partial\Omega.
    \end{aligned}\right.
\end{equation}
We will consider the above equation \eqref{eqn:Poisson} for the error analysis.
\begin{lemma}\label{Estm1}
Let $u$ be the classical solution of \eqref{eqn:Poisson}. Let $u_{\theta}\in \mathcal{U}$, then the following estimate holds
\begin{align*}
 \norm{u-u_{\theta}}^2_{L^2(\Omega)}\leq c(\alpha) \mathcal{L}(u_{\theta}).
\end{align*}
\end{lemma}
\begin{proof}
The following stability estimate for \eqref{eqn:Poisson} follow form the proof of \cite[Theorem 4.2]{Berggren:2004}, 
\begin{equation}\label{eqn:stab-aux}
\norm{u}_{L^2(\Omega)}\leq c(\norm{f}_{L^2(\Omega)}+\norm{g}_{L^2(\partial\Omega)}).
\end{equation}
Let $e_u =u-u_{\theta}$. Then $e_u$ satisfy the following problem
\begin{align}
&\left\{\begin{aligned}
    -\Delta e_u &=f+\Delta u_{\theta}\quad\text{in}\; \Omega,\\
     e_u &= g-u_\theta \quad \text{on}\; \partial\Omega,
\end{aligned}\right.\label{AuxTilde1}
\end{align}

Using the estimate \eqref{eqn:stab-aux} for the equation \eqref{AuxTilde1} we obtain

\begin{equation}\label{eqn:stab-aux1}
\norm{e_u}_{L^2(\Omega)}\leq c(\norm{f+\Delta u_\theta}_{L^2(\Omega)}+\norm{u_\theta - g}_{L^2(\partial\Omega)}).
\end{equation}

Taking the square on both sides of the \eqref{eqn:stab-aux1}, and introducing $\alpha$ we obtain the desired estimate.  
\end{proof}

The next result gives one crucial error decomposition. The errors $\mathcal{E}_{app}$, $\mathcal{E}_{stat}$ and
$\mathcal{E}_{opt}$ refer to the approximation error (due to restricting the space $H^2(\Omega)$ to the set of
neural network functions), the statistical error (due to approximating the continuous integrals via Monte Carlo
methods) and the optimization error (arising from the approximate local minimizer instead of a global minimizer).
The rest of the analysis focuses on the approximation and statistical errors. We ignore the optimization error,
whose analysis is notoriously challenging and is completely open.
\begin{theorem}\label{ErrEstm}
Let $u$ be the continuous solution and $u^{\mathcal{S}}_{\theta}$ be the solution of the problem \eqref{PINN:formulation} generated by a random solver $\mathcal{S},$ and $u_{\theta^*}$ be the global minimizer to the
loss $\widehat {\mathcal{L}}(u_{\theta})$. Then the following estimate holds
\begin{align}\label{error_form}
   \norm{u-u^{\mathcal{S}}_{\theta}}^2_{L^2(\Omega)}\leq c(\alpha)(\mathcal{E}_{app}+\mathcal{E}_{sta}+\mathcal{E}_{opt}),
\end{align}
where the errors $\mathcal{E}_{app}$, $\mathcal{E}_{stat}$ and $\mathcal{E}_{opt}$ are defined respectively by
\begin{align*}
\mathcal{E}_{app}&= \inf_{u_{\theta}\in \mathcal{U}} \norm{u-u_{\theta}}^2_{H^2(\Omega)},\quad
\mathcal{E}_{stat}= \sup_{u_{\theta}\in \mathcal{U}} [\mathcal{L}(u_{\theta})-\widehat{\mathcal{L}}(u_{\theta})] + \sup_{u_{\theta}\in \mathcal{U}} [\widehat{\mathcal{L}}(u_{\theta})-\mathcal{L}(u_{\theta})],\\
\mathcal{E}_{opt} &= \widehat{\mathcal{L}}(u^{\mathcal{S}}_{\theta})-\widehat{\mathcal{L}}(u_{\theta^*}).
\end{align*}
\end{theorem}
\begin{proof}
It follows directly that for any $u_{\theta}\in \mathcal{U}$
\begin{align*}
    \mathcal{L}(u^{\mathcal{S}}_{\theta})-\mathcal{L}(u)
   =&\mathcal{L}(u^{\mathcal{S}}_{\theta})-\widehat{\mathcal{L}}(u^{\mathcal{S}}_{\theta})+\widehat{\mathcal{L}}(u^{\mathcal{S}}_{\theta})-\widehat{\mathcal{L}}(u_{\theta^*}) +\widehat{\mathcal{L}}(u_{\theta^*}) -\widehat{\mathcal{L}}({u}_{\theta})\\ &+\widehat{\mathcal{L}}({u}_{\theta})-\mathcal{L}({u}_{\theta}) +\mathcal{L}({u}_{\theta})-\mathcal{L}(u)\\
\leq&\sup\limits_{u_{\theta}\in \mathcal{U}} [\mathcal{L}(u_{\theta})-\widehat{\mathcal{L}}(u_{\theta})]+\widehat{\mathcal{L}}(u^{\mathcal{S}}_{\theta})- \widehat{\mathcal{L}}(u_{\theta^*})\\
&+\sup_{u_{\theta}\in \mathcal{U}} [\widehat{\mathcal{L}}(u_{\theta})-\mathcal{L}(u_{\theta})]+\mathcal{L}(u_{\theta})-\mathcal{L}(u),
\end{align*}
where the inequality is due to the minimizing property of $u_{\theta^*}$, i.e.,
$\widehat{\mathcal{L}}(u_{\theta^*}) -\widehat{\mathcal{L}}({u}_{\theta})\leq 0$.
Now, by the trace inequality
\begin{equation*}
\norm{u_{\theta}-u}_{L^2(\partial\Omega)}\leq c \norm{u_{\theta}-u}_{H^2(\Omega)},
\end{equation*}
 since $\mathcal{L}(u)=0$, we deduce
\begin{align*}
 &\mathcal{L}(u_{\theta})-\mathcal{L}(u)=  \mathcal{L}(u_{\theta})\\
=& \norm{\Delta u_{\theta}+f}^2_{L^2(\Omega)}+\alpha \norm{u_{\theta}-g}^2_{L^2(\partial\Omega)}\\
    =& \norm{\Delta (u_{\theta}-u)}^2_{L^2(\Omega)}+\alpha \norm{u_{\theta}-u}^2_{L^2(\partial\Omega)}\\
    \leq& c(\alpha) \norm{u-u_{\theta}}^2_{H^2(\Omega)}.
\end{align*}
Consequently, we have
\begin{align}\label{Estm2}
     \mathcal{L}(u^{\mathcal{S}}_{\theta})-\mathcal{L}(u)\leq&\sup_{u_{\theta}\in \mathcal{U}} [\mathcal{L}(u_{\theta})-\widehat{\mathcal{L}}(u_{\theta})]+\widehat{\mathcal{L}}(u^{\mathcal{S}}_{\theta})-\widehat{\mathcal{L}}(u_{\theta^*})+\sup_{u_{\theta}\in \mathcal{U}} [\widehat{\mathcal{L}}(u_{\theta})-\mathcal{L}(u_{\theta})]\nonumber\\
     &+c(\alpha) \inf_{u_{\theta}\in \mathcal{U}} \norm{u-u_{\theta}}^2_{H^2(\Omega)}.
\end{align}
It follows from Lemma \ref{Estm1} and \eqref{Estm2} that
\begin{align*}
     \norm{u-u^{\mathcal{S}}_{\theta}}^2_{L^2(\Omega)}\leq& c(\alpha) \Bigl[\sup_{u_{\theta}\in \mathcal{U}} [\mathcal{L}(u_{\theta})-\widehat{\mathcal{L}}(u_{\theta})]
    +\widehat{\mathcal{L}}(u^{\mathcal{S}}_{\theta})-\widehat{\mathcal{L}}(u_{\theta^*})\\
    &+\sup_{u_{\theta}\in \mathcal{U}} [\widehat{\mathcal{L}}(u_{\theta})-\mathcal{L}(u_{\theta})]+ \inf_{u_{\theta}\in \mathcal{U}} \norm{u-u_{\theta}}^2_{H^2(\Omega)}\Bigr].
\end{align*}
This proves the desired assertion.
\end{proof}

\subsection{Approximation Error}\label{subsec:approxerr}
We recall the following result on the approximation error \cite[Proposition 4.8]{GuhrinRaslan:2021}.
\begin{proposition}\label{prop:approx}
    Let $p\geq 1,m,k,d\in \mathbb{N}^{+}, m\geq k+1.$ let $\rho$ be the logistic or $\tanh$ function. For any $\epsilon>0$ and $y\in W^{m,p}([0,1]^d)$ with $\norm{y}_{W^{m,p}([0,1]^d)}\leq 1,$ there exists a neural network $y_{\theta}$ with depth $c\log(d+m)$ and $c(d,m,p,k)\epsilon^{-\frac{d}{m-k-\mu}}$ non-zero weights such that
    \begin{align*}
        \norm{y-y_{\theta}}_{W^{k,p}([0,1]^d)}\leq \epsilon,
    \end{align*}
    where $\mu$ is an arbitrarily small positive number.    Moreover, the weights $\theta$ in the neural network $y_\theta$ are bounded in absolute value by
$c(d,m,p,k)\epsilon^{-2-\frac{2(d/p+d+k+\mu)+d/p+d}{m-k-\mu }}.$
\end{proposition}

Then we have the following bound on the approximation error $\mathcal{E}_{app}$.
\begin{lemma}\label{lem:approx}
Fix a tolerance $\epsilon\in (0,1)$. Let the optimal solution $u\in H^3(\Omega)$, then there exist neural network
functions $ u_{\theta}\in \mathcal{U} $ with
$\mathcal{U} = \mathcal{N}_{\rho}\Bigl(c\log(d+3),c(d)\epsilon^{-\frac{d}{1-\mu}},c(d)\epsilon^{-\frac{9d+12}{2-2\mu}}\Bigr)$ such that
$$\mathcal{E}_{app}\leq c(\norm{u}_{H^3(\Omega)}) \epsilon^2,$$
where the constant $c$ depends on the $H^3(\Omega)$ regularity of $u$.
\end{lemma}
\begin{proof}
Indeed, we have
\begin{align*}
 \inf\limits_{u_{\theta}\in \mathcal{U}}\norm{u-u_{\theta}}^2_{H^2(\Omega)}&=\norm{u}^2_{H^3(\Omega)}\inf_{u_{\theta}\in \mathcal{U}}\norm{\frac{u}{\norm{u}_{H^3(\Omega)}}-\frac{u_{\theta}}{\norm{u}_{H^3(\Omega)}}}^2_{H^2(\Omega)}
 =\norm{u}^2_{H^3(\Omega)}\inf\limits_{u_{\theta}\in \mathcal{U}}\norm{\frac{u}{\norm{u}_{H^3(\Omega)}}-u_{\theta}}^2_{H^2(\Omega)}.
\end{align*}
By Proposition \ref{prop:approx}, there exists a neural network function $u_{\theta}\in \mathcal{U}= \mathcal{N}_{\rho}(c\log(d+3),c(d)\epsilon^{-\frac{d}{1-\mu}},c(d)\epsilon^{-\frac{9d+12}{2-2\mu}})$
such that
$\norm{\frac{u}{\norm{u}_{H^3(\Omega)}}-u_{\theta}}_{H^2(\Omega)}\leq \epsilon$. Hence,
\begin{align*}
    \mathcal{E}_{app}&= \inf\limits_{u_{\theta}\in \mathcal{U}} \norm{u-u_{\theta}}^2_{H^2(\Omega)}
    \leq c(\norm{u}_{H^3(\Omega)}) \epsilon^2.
\end{align*}
Combining these results yields the desired assertion.
\end{proof}

\subsection{Statistical errors}\label{subsec:staterr}
The error $\mathcal{E}_{stat}= \sup_{u_{\theta}\in \mathcal{U}} |\mathcal{L}(u_{\theta})-\widehat{\mathcal{L}}(u_{\theta})|$ arises from using the Monte Carlo method to approximate the 
integral over the domain / boundary. This error influences the minimizer and thus also the neural network approximations 
to the solution. To bound the error $\mathcal{E}_{stat}$, we define
\begin{align*}
 \Delta \mathcal{E}_{i} &= \sup_{u_{\theta}\in \mathcal{U}} |\Omega|\mathbb{E}_{U(\Omega)}\left[|\Delta u_{\theta}(X)+f(X)|^2\right] - \frac{|\Omega|}{n}
  \sum_{i=1}^{n}|\Delta u_{\theta}(X_i)+f(X_i)|^2,\\
 \Delta \mathcal{E}_{b} &= \sup_{u_{\theta}\in \mathcal{U}} |\partial\Omega|\mathbb{E}_{U(\partial\Omega)}\left[|u_{\theta}(Y)-g(Y)|^2\right] - \frac{|\partial\Omega|}{m}
  \sum_{j=1}^{m}|u_{\theta}(Y_j)-g(Y_j)|^2,
\end{align*}
Then by the triangle inequality, we have
\begin{align}\label{Estm:stat1}
  \mathcal{E}_{stat} \leq \Delta\mathcal{E}_{i}+ \alpha \Delta \mathcal{E}_{b}.
\end{align}

It remains to bound the terms $\Delta\mathcal{E}_{i}$ and $\Delta\mathcal{E}_{b}$ separately. We define the following function classes
\begin{align*}
  \mathcal{F}_{i}  &= \{l: \Omega\to \mathbb{R}|\,\, l(x) = (\Delta u_{\theta}(x)+f)^2, u_{\theta}\in \mathcal{U}\},\\
  \mathcal{F}_{b} & = \{l:\partial\Omega\to \mathbb{R}|\,\, l(x) = |u_{\theta}(x)-g(x)|^2, u_{\theta} \in \mathcal{U}\}.\\
\end{align*} 

Now we prove bounds on the statistical errors $\Delta\mathcal{E}_{i}$ and $\Delta\mathcal{E}_{b}$ using Rademacher complexity
 \cite{BartlettMendelson:2002}. Rademacher complexity measures the complexity of a collection of functions
by the correlation between function values with Rademacher random variables.
\begin{definition}\label{def: Rademacher}
Let $\mathcal{F}$ be a real-valued function class defined on the domain $\Omega$ {\rm(}or the boundary $\partial\Omega${\rm)} and $\xi=\{\xi_j\}_{j=1}^{n}$ be i.i.d. samples from the distribution $\mathcal{U}(\Omega)$ {\rm(}or the distribution $\mathcal{U}(\partial\Omega)${\rm)}. Then
	the Rademacher complexity $\mathcal{R}_n(\mathcal{F})$ of the class $\mathcal{F}$ is defined by
	\begin{equation*}
		\mathcal{R}_n(\mathcal{F})=\mathbb{E}_{\xi,\omega}\bigg{[}\sup_{v\in\mathcal{F}}\ \frac{1}{n}\ \sum_{j=1}^{n}\omega_j v(\xi_j)\  \bigg{]},
	\end{equation*}
	where $\omega=\{\omega_j\}_{j=1}^n$ are i.i.d  Rademacher random
	variables with probability
	$P(\omega_j=1)=P(\omega_j=-1)=\frac12$.
\end{definition}
By a standard symmetrization argument, we have the following bounds on $\Delta\mathcal{E}_{i}$ and $\Delta\mathcal{E}_{b}$
in terms of the Rademacher complexity of the sets $\mathcal{F}_{i}$ and $\mathcal{F}_{b}$. These bounds play a crucial 
role in deriving the final error estimate.

\begin{lemma}\label{Estm:sata1}
For the errors $\Delta\mathcal{E}_{i}$ and $\Delta\mathcal{E}_{b}$ the following bounds hold.
\begin{align*}
  \mathbb{E}_{\{X_k\}_{k=1}^{n}}[\Delta\mathcal{E}_{i}] &\leq 2\mathcal{R}_{n} (\mathcal{F}_{i}),\\
  \mathbb{E}_{\{Y_k\}_{k=1}^{m}}[\Delta\mathcal{E}_{b}] &\leq 2\mathcal{R}_{m} (\mathcal{F}_{b}).
\end{align*}
\end{lemma}
\begin{proof}
The proof is standard and is only included for completeness. We prove the first inequality here the second one follows similarly. Let $\{\tilde{X}_k\}_{k=1}^n$ as an independent copy of $\{X_k\}_{
k=1}^n$, and $h(x)=(\Delta u_{\theta}(x)+f(x))^2$. Then we have
\begin{align*}
  |\Omega| \mathbb{E}_{U(\Omega)}\left[h(X)\right] - \frac{|\Omega|}{n}
  \sum_{k=1}^{n}h(X_k)
 =& \frac{|\Omega|}{n}\mathbb{E}_{\{\tilde X_k\}_{k=1}^n}\sum_{k=1}^n h(\tilde X_k) - \frac{|\Omega|}{n}
  \sum_{k=1}^{n}h(X_k)\\
 = & \frac{|\Omega|}{n}\mathbb{E}_{\{\tilde X_k\}_{k=1}^n}\sum_{k=1}^n\left[h(\tilde{X}_k) - h(X_k)\right].
\end{align*}
Taking supremum in $ h \in\mathcal{H}_i $, applying Jensen's inequality, and
taking then expectation with respect to $\{X_k\}_{k=1}^n$ give
\begin{align*}
  \mathbb{E}_{\{X_k\}_{k=1}^n}[\Delta\mathcal{E}_{i}]
  \leq &\frac{\lvert\Omega\rvert}{n} \mathbb{E}_{\{X_k,\tilde{X}_k\}_{k=1}^n} \sup_{h \in\mathcal{H}_i} \sum_{k=1}^n \Bigl[h(\tilde{X}_k)-h(X_k)\Bigr].
\end{align*}
By the independence of $\{X_k\}_{k=1}^n$ and $\{\tilde X_k\}_{k=1}^n$, inserting the Rademacher
random variables $\sigma_k$ does not change the distribution, and hence
\begin{align*}
  \mathbb{E}_{\{X_k\}_{k=1}^n}[\Delta\mathcal{E}_{i}]
 \leq & \frac{\lvert\Omega\rvert}{n} \mathbb{E}_{\{X_k,\tilde{X}_k,\sigma_k\}_{k=1}^n} \sup_{h \in\mathcal{H}_i} \sum_{k=1}^n \sigma_k\Bigl[h(\tilde X_k)-h(X_k)\Bigr],
\end{align*}
Since $h(X_k)$ and $h(\tilde X_k)$ are independent, $\sigma_k h(\tilde{X}_k)$ and
$-\sigma_kh(X_k)$ have the same distribution, and thus we have
\begin{align*}
    \mathbb{E}_{\{X_k\}_{k=1}^n}[\Delta\mathcal{E}_{i}] &\leq \frac{|\Omega|}{n} \mathbb{E}_{\{\tilde{X}_k,\sigma_k\}_{k=1}^n} \sup_{h \in\mathcal{H}_i} \sum_{k=1}^n \sigma_k h(\tilde{X}_k)+\frac{|\Omega|}{n} \mathbb{E}_{\{X_k,\sigma_k\}_{k=1}^n} \sup_{h \in\mathcal{H}_i} \sum_{k=1}^n -\sigma_k h(X_k) \\
  & \leq \frac{2|\Omega|}{n} \mathbb{E}_{\{X_k,\sigma_k\}_{k=1}^n} \sup_{h \in\mathcal{H}_i} \sum_{k=1}^n \sigma_kh(X_k)
  \leq 2|\Omega|\mathbb{E}_{\{X_k,\sigma_k\}_{k=1}^n} \Bigl[ \sup_{h\in \mathcal{F}_{i}} \frac{1}{n} \sum_{k=1}^n \sigma_kh(X_k)\Bigr]\\
  &\leq 2|\Omega|\mathcal{R}_n(\mathcal{F}_{i}).
\end{align*}
This completes the proof of the lemma.
\end{proof}

 Now we bound the Rademacher complexity of the function classes $\Delta\mathcal{E}_{i}$ and $\Delta\mathcal{E}_{b}$. This is
achieved by combining Lipschitz continuity of DNN functions (or its derivatives) with the target function class
in the DNN parameters see \cite[Lemmas A.3 \& A.4]{OCP_PINN_Ramesh}, and Dudley's formula in Lemma
\ref{lem:Dudley} below. The next result gives the reduction to parameterization.
\begin{lemma}\label{lem:NN-Lip}
Let $\mathcal{U}=\mathcal{N}_\rho(L,\boldsymbol{n}_L,R)$, with depth $L$, $\boldsymbol{n}_L$ nonzero DNN parameters and maximum bound $R$. For the function classes $\mathcal{F}_{i}$ and $\mathcal{F}_{b}$ the functions are
uniformly bounded:
\begin{align*}
  \|h\|_{L^\infty(\Omega)} &\leq (dL\boldsymbol{n}_L^{2(L-1)}R^{2L}+\|f\|_{L^\infty(\Omega)})^2:=M_i &&\quad h\in \mathcal{F}_{i},\\
  \|h\|_{L^\infty(\partial\Omega)}& \leq (\boldsymbol{n}_LR+\|g\|_{L^\infty(\partial\Omega)})^2:=M_b, && \quad h\in \mathcal{F}_{b}
\end{align*}
Moreover, the following Lipschitz continuity estimates in the DNN parameters hold
\begin{align*}
  \|h-\tilde h\|_{L^\infty(\Omega)} & \le  \Lambda_{i}\|\theta -\tilde\theta\|_{\ell^2} ,\quad \forall h,\tilde h\in \mathcal{F}_{i},\\
  \|h-\tilde h\|_{L^\infty(\partial\Omega)} & \leq \Lambda_{b}\|\theta-\tilde \theta\|_{\ell^2},\quad \forall h,\tilde h\in \mathcal{F}_{b},
\end{align*}
with the Lipschitz constants given by
\begin{align*}
  \Lambda_{i} & =4dL^2\eta \sqrt{\boldsymbol{n}_L}\boldsymbol{n}_L^{3L-3}R^{3L-3}\big(dL \boldsymbol{n}_L^{2(L-1)}R^{2L}+ \|f\|_{L^\infty(\Omega)}\big),\\
  \Lambda_{b} & = 2 \sqrt{\boldsymbol{n}_L}\boldsymbol{n}_L^{L-1}R^{L-1}(\boldsymbol{n}_LR+\norm{g}_{L^{\infty}(\partial\Omega)}).
\end{align*}
\end{lemma}
\begin{proof}
The assertions follow directly from \cite[Lemma A.2 \& A.3]{OCP_PINN_Ramesh}.
Indeed, for $h_{\theta}\in \mathcal{F}_{i}$, we have
\begin{align*}
  |h_{\theta}(x)| &\leq \Big(\sum_{i=1}^d\|\partial_{x_i}^2 u_{\theta}\|_{L^2(\Omega)}+\|f\|_{L^\infty(\Omega)}\Big)^2\\
 &\leq (dL\boldsymbol{n}_L^{2(L-1)}R^{2L}+\|f\|_{L^\infty(\Omega)})^2.
\end{align*}
For any $h_{\theta}, h_{\tilde\theta}\in \mathcal{F}_{i}$, by completing the squares and Cauchy-Schwarz inequality, we have
\begin{align*}
    &| h_{\theta}(x) - h_{\tilde\theta}(x)| = |(\Delta u_{\theta}(x)+f(x))^2-
  (\Delta u_{\tilde\theta}(x)+f(x))^2|\\
    =& |(\Delta u_{\theta}(x)+f(x)+\Delta u_{\tilde\theta}(x)+f(x))(\Delta u_{\theta}(x)-
  \Delta u_{\tilde\theta}(x))|\\
  \leq & 2\big(dL \boldsymbol{n}_L^{2(L-1)}R^{2L}+ \|f\|_{L^\infty(\Omega)}\big) \|\Delta u_{\theta}(x)-\Delta u_{\tilde \theta}(x)\|_{L^\infty(\Omega)}.
\end{align*}
Then by \cite[Lemma A.2, A.3, \& A.4]{OCP_PINN_Ramesh}, we deduce
\begin{align*}
      |h_{\theta}(x) - h_{\tilde\theta}(x)|
      \leq &
 2\big(dL \boldsymbol{n}_L^{2(L-1)}R^{2L}+ \|f\|_{L^\infty(\Omega)}\big)\times\big(2dL^2\eta \sqrt{\boldsymbol{n}_L}\boldsymbol{n}_L^{3L-3}R^{3L-3}\|\theta-\tilde \theta\|_{\ell^2}\big).
\end{align*}
The remaining estimates follow similarly. This completes the proof of the lemma.
\end{proof}

Next we bound the Rademacher complexities $\mathcal{R}_n(\mathcal{F}_{i})$ and $\mathcal{R}_m(\mathcal{F}_{b})$ using the concept
of the covering number. Let $\mathcal{F}$ be a real-valued function class equipped with the metric $\rho$. An
$\epsilon$-cover of the class $\mathcal{F}$ with respect to the metric $\rho$ is a collection
of points $\{f_i\}_{i=1}^n \subset \mathcal{F}$ such that for every $f\in \mathcal{F}$, there
exists at least one $i \in \{1,\dots,n\}$ such that $\rho(f, f_i) \leq \epsilon$. The
$\epsilon$-covering number $\mathcal{C}(\mathcal{F}, \rho, \epsilon)$ is the minimum
cardinality among all $\epsilon$-cover of the class $\mathcal{F}$ with respect to the metric
$\rho$. Then we can state the well-known Dudley's theorem \cite[Theorem 9]{LuLu:2021priori}
and \cite[Theorem 1.19]{wolf2018mathematical}.
\begin{lemma}\label{lem:Dudley}
Let $M_\mathcal{F}:=\sup_{f\in\mathcal{F}} \|f\|_{L^{\infty}(\Omega)}$, and $\mathcal{C}(\mathcal{F},\|\cdot\|_{L^{\infty}(\Omega)},\epsilon)$ be the covering number of the set $\mathcal{F}$. Then the Rademacher complexity $\mathcal{R}_n(\mathcal{F})$ is bounded by
\begin{equation*}
\mathcal{R}_n(\mathcal{F})\leq\inf_{0<s< M_\mathcal{F}}\bigg(4s\ +\ 12n^{-\frac12}\int^{M_\mathcal{F}}_{s}\big(\log\mathcal{C}(\mathcal{F},\|\cdot\|_{L^{\infty}(\Omega)},\epsilon)\big)^{\frac12}\ {\rm d}\epsilon\bigg).
\end{equation*}
\end{lemma}

Now we can state the bound on the statistical errors.
\begin{theorem}\label{RademacherClx}
For the function classes $\mathcal{F}_{i}$ and $\mathcal{F}_{b}$ the following bounds hold
\begin{align*}
\mathcal{R}_n(\mathcal{F}_{i}) & \leq c n^{-\frac12}\boldsymbol{n}_L^{4L-\frac{7}{2}}R^{4L}(\log^\frac12 R +\log^\frac12 \boldsymbol{n}_L+\log^\frac12 n),\\
\mathcal{R}_m(\mathcal{F}_{b})&\leq \tilde{c}m^{-\frac12}R^2 \boldsymbol{n}_L^{\frac52}\big(\log^\frac12 R+\log^\frac12 \boldsymbol{n}_L+\log^\frac12 m\big),\\
 \end{align*}
 where the constants $c$ depends on $d$ and $\|f\|_{L^\infty(\Omega)}$ and $\tilde{c}$ depends on $d$ and $\|g\|_{L^\infty(\partial\Omega)}$ at most polynomially.
\end{theorem}
\begin{proof}
Recall that for any $m \in \mathbb{N }$, $r \in [1, \infty)$, $\epsilon \in  (0,1)$,
and $ B_r := \{x\in\mathbb{R}^m:\ \|x\|_{\ell^2}\leq r\}$, then by a simple counting
argument (see, e.g., \cite[Proposition 5]{CuckerSmale:2002} or \cite[Lemma 5.5]{JiaoLai:2021error}), we have
\begin{equation*}
	\log N(\epsilon,B_r,\|\cdot\|_{\ell^2})\leq m\log (2r\sqrt{m}\epsilon^{-1}).
\end{equation*}
It follows directly from the Lipschitz continuity of NN functions and the estimate $\|\theta\|_{\ell^2}\leq \sqrt{\boldsymbol{n}_L}\|\theta\|_{\ell^\infty}\leq \sqrt{\boldsymbol{n}_L}R$ that
\begin{align*}
&\log \mathcal{C}(\mathcal{F}_{i},\|\cdot\|_{L^{\infty}(\Omega)},\epsilon)\leq \log \mathcal{C}(\mathcal{N}_Y,\|\cdot\|_{\ell^2},\Lambda_{i}^{-1}\epsilon)\leq \boldsymbol{n}_L\log(2\boldsymbol{n}_LR\Lambda_{i}\epsilon^{-1}),
\end{align*}
where $\mathcal{N}_Y$ denotes the parameter space for $\mathcal{U},$ and $\Lambda_{i}:=4dL^2\eta \sqrt{\boldsymbol{n}_L}\boldsymbol{n}_L^{3L-3}R^{3L-3}\big(dL \boldsymbol{n}_L^{2(L-1)}R^{2L}+ \|f\|_{L^\infty(\Omega)}\big)$,
cf. Lemma \ref{lem:NN-Lip}. By Lemma \ref{lem:NN-Lip}, we also have $M_{i}:=(dL\boldsymbol{n}_L^{2(L-1)}R^{2L}+\|f\|_{L^\infty(\Omega)})^2$. Then letting $s=n^{-\frac12}$
in Lemma \ref{lem:Dudley} gives
\begin{align*}
    &\quad \mathcal{R}_n(\mathcal{F}_{i})\leq4n^{-\frac12}+12n^{-\frac12}\int^{M_{i}}_{n^{-\frac12}}{\big(\boldsymbol{n}_L\mbox{log}(2R\boldsymbol{n}_L\Lambda_{i}\epsilon^{-1})\big)}^{\frac12}\ {\rm d}\epsilon\\
			&\leq4n^{-\frac12}+12n^{-\frac12}M_{i}{\big(\boldsymbol{n}_L \mbox{log}(2R\Lambda_{i}\boldsymbol{n}_L \sqrt{n})\big)}^\frac12 \\&
			 \leq4n^{-\frac12}+24n^{-\frac12}(dL\boldsymbol{n}_L^{2(L-1)}R^{2L}+\|f\|_{L^\infty(\Omega)})^2 \boldsymbol{n}_L^{\frac12} \big(\log R + \log\boldsymbol{n}_L+\log \Lambda_{i}+\frac12\log n+ \log 2\big)^\frac12.
\end{align*}
Since $1\leq R$ and $2\leq L\leq c\log(d+3)$ (due to Proposition \ref{prop:approx}), we can bound the term $\log\Lambda_{i}$ by
\begin{equation*}
	\log\Lambda_{i}\leq c(\log R +\log \boldsymbol{n}_L +\tilde c),
\end{equation*}
with the constants $c$ and $\tilde c$ depending on $d$, $L$ and $\|f\|_{L^\infty(\Omega)}$ at most polynomially.
Hence, we have
\begin{equation*}
   \mathcal{R}_n(\mathcal{F}_{i})\leq c n^{-\frac12}\boldsymbol{n}_L^{4L-\frac{7}{2}}R^{4L}(\log^\frac12 R +\log^\frac12 \boldsymbol{n}_L+\log^\frac12 n),
\end{equation*}
where $c>0$ depends on $d$ and $\|f\|_{L^\infty(\Omega)}$ at most polynomially.
The rest estimates follow similarly,
where the involved constant depend on $d$ and $\|g\|_{L^\infty(\partial\Omega)}$ at most polynomially.
\end{proof}

 Hence, using \eqref{Estm:stat1}, Lemma \ref{Estm:sata1}, and Theorem \ref{RademacherClx}, we have the following estimate for statistical error.
\begin{lemma}\label{Estm:stat} 
The statistical error
    \begin{align*}
      \mathcal{E}_{stat}\leq c \brb{n^{-\frac12}\boldsymbol{n}_L^{4L-\frac{7}{2}}R^{4L}(\log^\frac12 R +\log^\frac12 \boldsymbol{n}_L +\log^\frac12 n)+m^{-\frac12}R^2 \boldsymbol{n}_L^{\frac52}\big(\log^\frac12 R+\log^\frac12 \boldsymbol{n}_L+\log^\frac12 m\big)}. 
    \end{align*}
     where the constant $c$ depends on $d$ and $\|f\|_{L^\infty(\Omega)}$ and  $\|g\|_{L^\infty(\partial\Omega)}$ at most polynomially.
\end{lemma}

With the preparation in the last two subsections on the bounds of approximation and statistical errors, we state the main result in the following.
\begin{theorem}\label{thm:main}
    Let $\epsilon>0$ be a given tolarance. Let $\rho$ be logistic function $\frac{1}{1+e^{-x}}$ or $\tanh$ function $\frac{e^x-e^{-x}}{e^x+e^{-x}}.$ Let $u\in H^3(\Omega)$ and $u_{\theta}^{\mathcal{S}}$ be the solution of the problems \eqref{eqn:Poisson0} and \eqref{PINN:formulation} respectively. Then for any $\mu\in(0,1)$ set the parametrized neural network function classes 
    $\mathcal{U} = \mathcal{N}_{\rho}\Bigl(c\log(d+3),c(d)\epsilon^{-\frac{d}{1-\mu}},c(d)\epsilon^{-\frac{9d+12}{2-2\mu}}\Bigr),$ and the number of interior sample points $n=c(\alpha,d)\epsilon^{-4-\frac{(44d+48)\log(d+3)-7d}{1-\mu}}$ and number of boundary sample points $m=c(\alpha,d)\epsilon^{-4-\frac{23d+24}{1-\mu}},$ and if the optimization error $\mathcal{E}_{opt}\leq \epsilon^2.$ Then 
\begin{align}\label{EstmDesired}
    \norm{u-u^{\mathcal{S}}_{\theta}}_{L^2(\Omega)}\leq c \epsilon.
\end{align}
where the constant $c$ depends on $\alpha$, $d$, $\|f\|_{L^\infty(\Omega)}$, $\|g\|_{L^\infty(\partial\Omega)}$ and $\|u\|_{H^3(\Omega)}$.

\end{theorem}

\begin{proof}
    From Theorem \ref{ErrEstm}, we have
    \begin{align}\label{Estm_erruse}
   \norm{u-u^{\mathcal{S}}_{\theta}}^2_{L^2(\Omega)}\leq c(\alpha)(\mathcal{E}_{app}+\mathcal{E}_{sta}+\mathcal{E}_{opt}).
\end{align}
From Lemma \ref{lem:approx}, there exists neural network function classes  $\mathcal{U} = \mathcal{N}_{\rho}\Bigl(c\log(d+3),c(d)\epsilon^{-\frac{d}{1-\mu}},c(d)\epsilon^{-\frac{9d+12}{2-2\mu}}\Bigr),$ such that
\begin{align}\label{EstmAppUse}
\mathcal{E}_{app}\leq c(\norm{u}_{H^3(\Omega)}) \epsilon^2.
\end{align}
Now in Lemma \ref{Estm:stat} we can put $L=c\log(d+3), \boldsymbol{n}_L=c(d)\epsilon^{-\frac{d}{1-\mu}},R=c(d)\epsilon^{-\frac{9d+12}{2-2\mu}}$ and setting $n=c(\alpha,d)\epsilon^{-4-\frac{(44d+48)\log(d+3)-7d}{1-\mu}}$ and $m=c(\alpha,d)\epsilon^{-4-\frac{23d+24}{1-\mu}},$ we obtain 
\begin{align}\label{EstmStatUse}
   \mathcal{E}_{stat}\leq c \epsilon^2. 
\end{align}
Combining \eqref{Estm_erruse}, \eqref{EstmAppUse}, and \eqref{EstmStatUse} we obtain the desired estimate \eqref{EstmDesired}. 
\end{proof}

In the next three remarks, we discuss the error estimates for the other three methods(DRM, VPINN, and WAN). The error estimate for DRM is quite similar to the PINN. However, the error analysis for VPINN and WAN is much more involved and needs many technicalities. Therefore we skip the proofs of all other three methods, and an interested reader can find those in the corresponding references given below.

\begin{remark}\label{rem:error_estm_DRM}
    In this remark we discuss an error estimate result for DRM solution from \cite[Section 6]{JiaoLai:2021error}. The error analysis for DRM can be derived similarly to the PINN method (described earlier in this Section). We choose $\mathcal{A}=\mathbb{I}_{d \times d}$, $\bm{\beta}=\mathbf{0},$ $g=0$ in the equation \eqref{eqn:Poisson0}. We state the error estimate result for DRM from Section 6 of \cite{JiaoLai:2021error} as follows: Let $u\in H^2(\Omega)$ and $u^{\mathcal{S}}_\theta$ be the continuous and neural network solutions obtained from the problems \eqref{eqn:Poisson0} and \eqref{drm_emploss} respectively. Then for any $\mu\in(0,1),$ set the parametrized neural network function classes 
    $\mathcal{U} = \mathcal{N}_{\rho}\Bigl(c\log(d+1),c(d)\epsilon^{-\frac{5d}{2(1-\mu)}},c(d)\epsilon^{-\frac{45d+40}{4-4\mu}}\Bigr),$ and the number of interior and boundary sample points $n=m=c(d,\alpha,\Omega)\epsilon^{-\frac{cd\log(d+1)}{1-\mu}}$(with $c$ some positive constant), and let the optimization error $\mathcal{E}_{opt}\leq \epsilon^2.$ Then the following $H^1$ error estimates of the solution holds:
    \begin{align*}
    \norm{u-u^{\mathcal{S}}_{\theta}}_{H^1(\Omega)}\leq c(d,\alpha,\Omega,\|f\|_{L^\infty(\Omega)},\| u\|_{H^2(\Omega)}) \epsilon.
\end{align*}
\end{remark}

\begin{remark}\label{rem:error_estm_VPINN}
    This remark concerns the a prior error estimates for the VPINN solution. For more details of the formulations and proofs of error estimates, we refer to the article \cite{berrone2022variational}. In this article, the authors slightly modify the loss function and propose the VPINN as follows:
\begin{align}\label{loss_modified_vpinn}
   \underset{w_\theta}{\text{min}}\quad \mathcal{L}_h(w_\theta) = \sum_{i=1}^{n_b} \alpha_i r_{h,i}^2(\mathcal{I}_h w_\theta),
\end{align}
where $r_{h,i}$ are defined in \eqref{residuals} and $\mathcal{I}_h w_\theta$ be an interpolation of $w_\theta$ onto the space $V_h$(defined in subsection \ref{subsec:vpinn}). We state the error estimate (see \cite[Theorem 6]{berrone2022variational}) for VPINN solution as follows: Let $u\in H^{k+1}(\Omega)$ be the weak solution of \eqref{eqn:Poisson0} and, $u_\theta$ be a neural network solution of VPINN obtained from \eqref{loss_modified_vpinn}. Then the following error estimate holds: 
\begin{align*}
  \norm{u-\mathcal{I}_h u_\theta}_{1,\Omega}\lesssim \left( 1+h^{-1}\right) h^k \left( \norm{f}_{L^2(\Omega)} + \norm{u}_{H^{k+1}(\Omega)} \right), 
\end{align*}
where $h$ is the mesh-size of the mesh $\mathcal{T}_h.$ The presence of the $h^{-1},$ which originates from the loss function, makes the estimate sub-optimal. 
\end{remark}

\begin{remark}\label{rem:error_estm_WAN}
    In this remark, we discuss the error analysis for WAN method. Due to the formulation of the WAN method \eqref{wan_minimax}, an error estimate of the form \eqref{error_form} is still unknown. We follow the error estimates for WAN method from the article \cite{bertoluzza2024best}. We select the functions $c$ and $\bm{\beta}$ in \eqref{eqn:Poisson0} such that $c-\nabla \cdot \bm{\beta}/2\geq 0,$ for the existence of a weak solution. We choose $g=0$ in \eqref{eqn:Poisson0} for simplicity of the analysis.  In the article \cite{bertoluzza2024best}, the authors derive a best approximation result for the WAN solution. The authors assume discrete inf-sup condition (see \cite[equation (28)]{bertoluzza2024best}) of the form
    \begin{align*}
        \kappa \norm{w_\theta}_{H^1(\Omega)} \leq \sup_{v_\eta \in V_\eta,v_\eta \neq 0} \frac{\mathcal{A}(w_\theta,v_\eta)}{\norm{v_\eta}_V} \quad \forall w_\theta \in W_\theta \cup S_\theta,
    \end{align*}
    where $W_\theta\subseteq H_0^1(\Omega)$ and $V_\eta\subseteq H_0^1(\Omega)$ are trial and test neural network spaces and $\kappa>0$. The set $S_\theta$ is a special function class defined in \cite[equation (17)]{bertoluzza2024best}. With the assumption of the above type inf-sup condition, the authors derive a best approximation result (see \cite[Lemma 3]{bertoluzza2024best}), as follows: 
    \begin{align*}
        \norm{u-u^*_\theta}_{H^1(\Omega)} \lesssim \inf_{w_\theta \in W_{\theta}} \norm{u-w_\theta}_{H^1(\Omega)}.
    \end{align*}
\end{remark}

    

\section{Acknowledgements} The authors extend their gratitude to Professor Neela Nataraj for proposing the idea of writing this review article on neural network-based solvers for PDEs, as well as for providing valuable comments and feedback.

\bibliographystyle{plain}
\bibliography{main}

\end{document}